\documentclass[a4paper, 12pt]{article}
\usepackage{amsmath}
\usepackage{amsfonts}
\usepackage{amssymb}
\usepackage{amscd}
\usepackage{amsthm}

\renewcommand{\P}{\mathbb{P}}
\newcommand{\Z}{\mathbb{Z}}
\newcommand{\Q}{\mathbb{Q}}
\newcommand{\R}{\mathbb{R}}
\newcommand{\C}{\mathbb{C}}
\newcommand{\Cb}{\mathbb{C}}

\newcommand{\F}{\mathcal{F}}
\renewcommand{\H}{\mathcal{H}}

\newcommand{\Hc}{\mathcal{H}}
\newcommand{\Ac}{\mathcal{A}}
\newcommand{\Ec}{\mathcal{E}}

\newcommand{\G}{\mathcal{G}}
\newcommand{\M}{\mathcal{M}}
\newcommand{\Pc}{\mathcal{P}}
\newcommand{\Qc}{\mathcal{Q}}
\newcommand{\Cc}{\mathcal{C}}

\newcommand{\Sc}{\mathcal{S}}

\newcommand{\K}{\mathcal{K}}

\newcommand{\Lc}{\mathcal{L}}

\newcommand{\Abc}{\mathcal{A}b}
\newcommand{\A}{\mathbf{A}}
\newcommand{\Ab}{\mathbb{A}}

\newcommand{\OO}{\mathcal{O}}

\newcommand{\Hom}{\mathrm{Hom}}
\newcommand{\ord}{\mathrm{ord}}

\newcommand{\Tame}{\mathrm{Tame}}
\newcommand{\colim}{\mathrm{colim}}
\newcommand{\Picard}{\mathrm{Picard}}
\newcommand{\Funt}{\mathrm{Fun}^+}
\newcommand{\Distr}{\mathrm{Distr}}

\newcommand{\codim}{\mathrm{codim}}
\newcommand{\Pic}{\mathrm{Pic}}
\newcommand{\Ext}{\mathrm{Ext}}
\newcommand{\RHom}{\mathrm{RHom}}
\newcommand{\Supp}{\mathrm{Supp}}

\newcommand{\Ker}{\mathrm{Ker}}

\newcommand{\rk}{\mathrm{rk}}

\newcommand{\Imm}{\mathrm{Im}}
\newcommand{\Spec}{\mathrm{Spec}}

\renewcommand{\dim}{\mathrm{dim}}

\renewcommand{\div}{\mathrm{div}}
\newcommand{\Div}{\mathrm{Div}}

\newcommand{\hhom}{\mathrm{hom}}

\newtheorem{theor}{Theorem}[section]
\newtheorem{lemma}[theor]{Lemma}
\newtheorem{prop}[theor]{Proposition}
\newtheorem{claim}[theor]{Claim}
\newtheorem{corol}[theor]{Corollary}

\theoremstyle{definition}
\newtheorem{defin}[theor]{Definition}
\newtheorem{examp}[theor]{Example}

\newtheorem{quest}[theor]{Question}

\theoremstyle{remark}
\newtheorem{rmk}[theor]{Remark}

\numberwithin{equation}{section}

\begin{document}

\title{Notes on the biextension of Chow groups}
\author{Sergey Gorchinskiy\footnote{The author was partially supported
by the grants RFBR 08-01-00095, Nsh-1987.2008.1, and INTAS
05-1000008-8118.}\\ \small{Steklov Mathematical Institute}\\
\small{\it e-mail: gorchins@mi.ras.ru}}
\date{}

\maketitle

\begin{abstract}
The paper discusses four approaches to the biextension of Chow groups
and their equivalences. These are the following: an
explicit construction given by S.\,Bloch, a construction in terms of the Poincar\'e
biextension of dual intermediate Jacobians, a construction in terms of
$K$-cohomology, and a construction in terms of determinant of cohomology
of coherent sheaves. A new approach to J.\,Franke's Chow categories is given.
An explicit formula for the Weil pairing of algebraic cycles is obtained.
\end{abstract}

\section{Introduction}

One of the questions about algebraic cycles is the following: what
can be associated in a bilinear way to a pair of algebraic cycles
$(Z,W)$ of codimensions $p$ and $q$, respectively, on a smooth
projective variety $X$ of dimension $d$ over a field $k$ with
$p+q=d+1$, i.e., what is an analogue of the linking number for
algebraic cycles? This question arose naturally from some of the
approaches to the intersection index of arithmetic cycles on
arithmetic schemes, i.e., to the height pairing for algebraic
cycles.

For homologically trivial cycles, a ``linking invariant'' was
constructed by S.\,Bloch in \cite{Blo84} and \cite{Blo89}, and independently by A.\,Beilinson
in~\cite{Bei87}. It turns
out that for an {\it arbitrary} ground field $k$, this invariant is no longer a number, but it is a
$k^*$-torsor, i.e., a set with a free transitive action of the
group~$k^*$. More precisely, in \cite{Blo89} a biextension $P$ of
$(CH^p(X)_{\hom},CH^q(X)_{\hom})$ by~$k^*$ is constructed, where
$CH^{p}(X)_{\hom}$ is the group of codimension $p$ homologically
trivial algebraic cycles on $X$ up to rational equivalence. This means that there is a map of sets
$$
\pi:P\to CH^p(X)_{\hom}\times CH^q(X)_{\hom},
$$
a free action of $k^*$ on the set $P$ such that $\pi$ induces a bijection
$$
P/k^*\cong CH^p(X)_{\hom}\times CH^q(X)_{\hom},
$$
and for all elements $\alpha,\beta\in CH^p(X)_{\hom}$,
$\gamma,\delta\in CH^q(X)_{\hom}$, there are fixed isomorphisms
$$
P_{(\alpha,\gamma)}\otimes P_{(\beta,\gamma)}\cong
P_{(\alpha+\beta,\gamma)}
$$
$$
P_{(\alpha,\gamma)}\otimes P_{(\alpha,\delta)}\cong
P_{(\alpha,\gamma+\delta)}
$$
such that certain compatibility axioms are satisfied (see \cite{Bre}). Here the
tensor product is taken in the category of $k^*$-torsors and
$P_{(*,*)}$ denotes a fiber of $P$ at $(*,*)\in
CH^p(X)_{\hom}\times CH^q(X)_{\hom}$. In other words, the
biextension $P$ defines a bilinear pairing between Chow groups of
homologically trivial cycles with value in the category of
$k^*$-torsors.

The biextension $P$ generalizes the Poincar\'e line bundle on the
product of the Picard and Albanese varieties. Note that if the ground field $k$ is a {\it
number field}, then each embedding of $k$ into its completion $k_v$
induces a trivialization of the biextension $\log|P|_v$ and the
collection of all these trivializations defines in a certain way the
height pairing for algebraic cycles (see \cite{Blo84}).

On the other hand, the biextension $P$ of
$(CH^p(X)_{\hom},CH^q(X)_{\hom})$ by $k^*$ for $p+q=d+1$ is an
analogue of the intersection index $CH^p\times CH^q(X)\to \Z$ for
$p+q=d$. There are several approaches to algebraic cycles, each of
them giving its own definition of the intersection index. A natural
question is to find analogous definitions for the biextension $P$.
The main goal of the paper is to give a detailed answer to this
question. Namely, we discuss four different constructions of
biextensions of Chow groups and prove their equivalences.

Let us remind several approaches to the intersection index and
mention the corresponding constructions of the biextension of Chow
groups that will be given in the article.

The most explicit way to define the intersection index is to use the moving lemma and the definition
of local multiplicities for proper intersections. Analogous to this
is the explicit definition of the biextension $P$ given
in~\cite{Blo89}.

If $k=\C$, one can consider classes of algebraic cycles in Betti
cohomology groups $H^{2p}_B(X(\C),\Z)$ and then use the product
between them and the push-forward map. Corresponding to this in
\cite{Blo89} it was suggested that for $k=\C$, the biextension $P$
should be equal to the pull-back via the Abel--Jacobi map of the
Poincar\'e line bundle on the product of the corresponding dual
intermediate Jacobians. This was partially proved in \cite{MS95} by
using the functorial properties of higher Chow groups and the
regulator map to Deligne cohomology.

A different approach uses the Bloch--Quillen formula
$CH^p(X)=H^p(X,\K_p)$, the product between cohomology groups of
sheaves, the product between $K$-groups, and the push-forward for $K$-cohomology.
Here $\K_p$ is the
Zariski sheaf associated to the presheaf given by the formula
$U\mapsto K_p(U)$ for an open subset $U\subset X$. A corresponding
approach to the biextension uses the pairing between complexes
$$
R\Gamma(X,\K_p)\times R\Gamma(X,\K_q)\to k^*[-d]
$$
for $p+q=d+1$.

Finally, one can associate with each cycle $Z=\sum_i n_iZ_i$ an
element $[\OO_Z]=\sum n_i[\OO_{Z_i}]\in K_0(X)$ and to use the
natural pairing on $K_0(X)$, i.e., to define the intersection index
by the formula ${\rm rk} R\Gamma(X,\OO_Z\otimes^L_{\OO_X}\OO_W)$.
Analogous to this one considers the determinant of cohomology
${\rm det} R\Gamma(X,\OO_Z\otimes^L_{\OO_X}\OO_W)$
to get a biextension of Chow groups. This is a generalization of what was done
for divisors on curves by P.\,Deligne in~\cite{Del}. With this aim a new approach to
J.\,Franke's Chow categories (see \cite{Fra}) is developed.
This approach uses a certain filtration ``by codimension of support'' on the Picard category
of virtual coherent sheaves on a variety (see \cite{Del}).

One interprets the compatibility of the first definition of the
intersection index with the second and the third one as the fact
that the cycle maps
$$
CH^p(X)\to H^{2p}_B(X(\C),\Z),
$$
$$
CH^p(X)\to H^{2p}(X,\K_p[-p])
$$
commute with products and push-forwards. Note that the cycle
maps are particular cases of canonical morphisms (regulators) from
motivic cohomology to various cohomology theories. This approach
explains quickly the comparison isomorphism between the
corresponding biextensions. Besides this we give a more explicit and
elementary proof of the comparison isomorphism in each case.

\begin{quest}
What should be associated to a pair of cycles of codimensions $p$, $q$ with $p+q=d+i$,
$i\ge 2$? Presumably, when $i=2$, one associates
a \mbox{$K_2(k)$-gerbe}.
\end{quest}

The paper has the following structure.
Sections~\ref{subsection-higherChow}-\ref{subsection-Abel-Jacobi}
contain the description of various geometric constructions that are necessary
for definitions of biextensions of Chow groups. In particular, in
Section~\ref{subsection-adelic} we recall several facts from
\cite{Gor07} about adelic resolution for sheaves of $K$-groups.

In Sections~\ref{subsection-quotientbiext} and
\ref{subsection-complexes} general algebraic constructions of
biextensions are discussed. In particular, we introduce the notion
of a bisubgroup and give an explicit construction of a biextension induced by
a pairing between complexes. Though these constructions are
elementary and general, the author could not find any reference for
them. As an example to the above notions, in
Section~\ref{subsection-Poincare} we consider the definition of the
Poincar\'e biextension of dual complex compact tori by $\Cb^*$.

Sections~\ref{subsection-acycles}-\ref{subsection-determbiext}
contain the constructions of biextensions of Chow groups according
to different approaches to algebraic cycles. In
Section~\ref{subsection-acycles} we recall from~\cite{Blo89} an
explicit construction of the biextension of Chow groups. This
biextension is interpreted in terms of the pairing between higher
Chow complexes (Proposition \ref{prop-Chowbiext}), as was suggested
to the author by S.\,Bloch. In Section~\ref{subsection-AJbiext} for
the complex base field, we consider the pull-back of the Poincar\'e
biextension of dual intermediate Jacobians
(Proposition~\ref{prop-AbelJacobi}).
Section~\ref{subsection-Kcohombiext} is devoted to the construction
of the biextension in terms of $K$-cohomology groups and a pairing
between sheaves of \mbox{$K$-groups} (Proposition
\ref{prop-biext-Kcohom}). We give an explicit description of this
biextension in terms of the adelic resolution for sheaves of
$K$-groups introduced in~\cite{Gor07}. In
Section~\ref{subsection-determbiext} we construct a filtration on the Picard
category of virtual coherent sheaves on a variety (Definition~\ref{defin-filtr})
and we define the biextension of
Chow groups in terms of the determinant of cohomology
(Proposition~\ref{prop-detcohombiext}). In each section we establish a canonical
isomorphism of the constructed biextension with the explicit
biextension from \cite{Blo89} described firstly. In
Sections~\ref{subsection-AJbiext} and \ref{subsection-Kcohombiext}
we give both explicit proofs and the proofs that use properties of
the corresponding regulator maps.
Finally, Section~\ref{subsection-weilpairing} gives an explicit formula for
the Weil pairing between torsion elements in $CH^*(X)_{\hom}$. This can be
considered as a generalization of the classical Weil's formula for
divisors on a curve. Also, the equivalence of different
constructions of biextensions of Chow groups implies the
interpretation of the Weil pairing in terms of a certain Massey
triple product.

The author is deeply grateful to A.\,A.\,Beilinson, S.\,Bloch,
A.\,M.\,Levin, A.\,N.\,Parshin, C.\,Soul\'e, and V.\,Vologodsky for
very stimulating discussions, and to the referee whose numerous
suggestions helped improving the paper very much. In particular, the
first proof of Proposition~\ref{prop-biext-Kcohom} was proposed by
the referee.

\section{Preliminary results}\label{section-prelim}

\subsection{Facts on higher Chow groups}\label{subsection-higherChow}

All varieties below are defined over a fixed ground field $k$.
For an equidimensional variety $S$, by $Z^p(S)$ denote the free abelian
group generated by all codimension $p$ irreducible subvarieties in $S$.
For an element $Z=\sum n_i Z_i\in Z^p(S)$,
by $|Z|$ denote the union of codimension $p$ irreducible subvarieties
$Z_i$ in $S$ such that $n_i\ne 0$.
Let $S^{(p)}$ be the set of all codimension $p$ schematic points on $S$.

Let us recall the definition of higher Chow groups (see \cite{Blo86}).
We put \mbox{$\Delta^n=\{\sum_{i=0}^n t_i=1\}\subset\Ab^{n+1}$}; note that
$\Delta^{\bullet}$ is a cosimplicial variety.
In particular, for each $m\le n$, there are
several face maps $\Delta^m\to\Delta^n$. For an equidimensional variety $X$,
by $Z^p(X,n)$ denote the free abelian
group generated by codimension $p$
irreducible subvarieties in $X\times \Delta^n$ that meet the subvariety
$X\times\Delta^m\subset X\times \Delta^n$ properly for any face $\Delta^m\subset
\Delta^n$. The simplicial group $Z^p(X,\bullet)$ defines a homological type complex;
by definition, $CH^p(X,n)=H_n(Z^p(X,\bullet))$ is the {\it higher Chow group} of $X$.
Note that $CH^p(X,0)=CH^p(X)$ and $CH^1(X,1)=k[X]^*$ for a regular variety $X$ (see op.cit.).
For a projective morphism of varieties $f:X\to Y$,
there is a push-forward morphism of complexes
$f_*:Z^p(X,\bullet)\to Z^{p+\dim(Y)-\dim(X)}(Y,\bullet)$.

For an equidimensional subvariety $S\subset X$, consider the subcomplex
\mbox{$Z_S^p(X,\bullet)\subset Z^p(X,\bullet)$} generated by elements from
$Z^p(X,n)$ whose support meets the subvariety
$S\times \Delta^m\subset X\times\Delta^n$ properly for any face
$\Delta^m\subset\Delta^n$. For a collection $\Sc=\{S_1\ldots,S_r\}$ of
equidimensional subvarieties in $X$, we put $Z^p_{\Sc}(X,\bullet)=
\cap^r_{i=1} Z^p_{S_i}(X,\bullet)$.
The following moving lemma is proven in Proposition 2.3.1 from \cite{Blo}
and in \cite{Lev}:

\begin{lemma}\label{lemma-moving}
Provided that $X$ is smooth over $k$ and either projective or
affine, the inclusion $Z_{\Sc}^p(X,\bullet)\subset Z^p(X,\bullet)$
is a quasiisomorphism for any $\Sc$ as above.
\end{lemma}

In particular, Lemma \ref{lemma-moving} allows to define
the multiplication morphism
$$
m\in \Hom_{D^-(\Abc)}(Z^p(X,\bullet)\otimes_{\Z}^L Z^q(X,\bullet),Z^{p+q}(X,\bullet)).
$$

Recall that a cycle $Z\in Z^p(X)$ is called {\it homologically trivial} if its class
in the \'etale cohomology group $H_{\acute e t}^{2p}(X_{\overline{k}},\Z_l(p))$
is zero for any prime $l\ne{\rm char}(k)$, where $\overline{k}$ is the algebraic
closure of the field $k$. Note that when ${\rm char}(k)=0$ the cycle $Z$ is
homologically trivial if and only if its class in the Betti cohomology group
$H^{2p}_B(X_{\C},\Z)$ is zero after we choose any model of $X$ defined over $\C$.
Denote by $CH^p(X)_{\hhom}$ the subgroup in $CH^p(X)$ generated by classes of
homologically trivial cycles.

The following result is proved in \cite{Blo89}, Lemma 1.

\begin{lemma}\label{lemma-FirstBloch}
Let $X$ be a smooth projective variety, $\pi:X\to \Spec(k)$ be the
structure morphism. Suppose that a cycle $Z\in Z^p(X)$ is
homologically trivial; then the natural homomorphism
$CH^{d+1-p}(X,1)\stackrel{m(-\otimes Z)}\longrightarrow
CH^{d+1}(X,1)\stackrel{\pi_*}\longrightarrow CH^1(k,1)=k^*$ is
trivial.
\end{lemma}

\begin{rmk}
The proof of Lemma \ref{lemma-FirstBloch} uses the regulator map from higher Chow groups
to Deligne cohomology if the characteristic is zero and to \'etale cohomology
if the characteristic is positive.
\end{rmk}

\begin{quest}
According to Grothendieck's standard conjectures, the statement of
Lemma \ref{lemma-FirstBloch} (at least up to torsion in $k^*$)
should be true if one replaces the homological triviality of the
cycle $W$ by the numerical one. Does there exist a purely algebraic
proof of this fact that does not use Deligne cohomology or \'etale
cohomology?
\end{quest}

\begin{rmk}
In Section
\ref{subsection-AJbiext} we give an analytic proof of Lemma \ref{lemma-FirstBloch} for
complex varieties, which uses only general facts from the Hodge theory
(see Lemma \ref{lemma-BlochAnal}).
\end{rmk}

\subsection{Facts on $K_1$-chains}\label{subsection-K1}

Let $X$ be an equidimensional variety over the ground field $k$.
We put $G^p(X,n)=\bigoplus\limits_{\eta\in X^{(p)}}K_n(k(\eta))$
(in this section we use these groups only for $n=0,1,2$). Elements of the group $G^{p-1}(X,1)$ are
called \mbox{\it $K_1$-chains}.
There are natural homomorphisms
$\Tame:G^{p-2}(X,2)\to G^{p-1}(X,1)$ and $\div:G^{p-1}(X,1)\to G^p(X,0)=Z^p(X)$.
Note that $\div\circ\Tame=0$. The subgroup $\Imm(\Tame)\subset G^{p-1}(X,1)$
defines an equivalence
on $K_1$-chains; we call this a \mbox{\it $K_2$-equivalence} on \mbox{$K_1$-chains}.
For a $K_1$-chain $\{f_{\eta}\}\in G^{p-1}(X,1)$, by $\Supp(\{f_{\eta}\})$
denote the union of codimension $p-1$ irreducible subvarieties $\overline{\eta}$ in
$X$ such that $f_{\eta}\ne 1$.

Let $S$ be an equidimensional subvariety; by $G^{p-1}_S(X,1)$ denote the group of
\mbox{$K_1$-chains} $\{f_{\eta}\}$ such that for any $\eta\in X^{(p-1)}$,
either $f_{\eta}=1$, or the closure $\overline{\eta}$ and the support
$|\div(f_{\eta})|$ meet $S$ properly. For a collection $\Sc=\{S_1\ldots,S_r\}$ of
equidimensional subvarieties in $X$, we put $G^{p-1}_{\Sc}(X,1)=
\cap^r_{i=1} G^{p-1}_{S_i}(X,1)$.

Define the homomorphism $N:Z^p(X,1)\to G^{p-1}(X,1)$ as follows.
Let $Y$ be an irreducible subvariety in $X\times\Delta^1$ that meets properly
both faces $X\times\{(0,1)\}$ and $X\times\{(1,0)\}$
(recall that $\Delta^1=\{t_0+t_1=1\}\subset\Ab^2$). By
$p_{X}$ and $p_{\Delta^1}$ denote the projections from $X\times\Delta^1$ to
$X$ and $\Delta^1$, respectively.
If the morphism $p_X:Y\to X$ is not generically finite onto its image,
then we put $N(Y)=0$. Otherwise, let
$\eta\in X^{(p-1)}$ be the generic point of $p_X(Y)$; we put
$f_{\eta}=(p_X)_*(p^*_{\Delta^1}(t_1/t_0))\in k(\eta)^*$ and $N(Y)=f_{\eta}\in G^{p-1}(X,1)$.
We extend the homomorphism $N$ to $Z^p(X,1)$ by linearity.

Conversely, given a point $\eta\in X^{(p-1)}$ and a rational
function $f_{\eta}\in k(\eta)^*$ such that $f_{\eta}\ne 1$, let
$\Gamma(f_{\eta})$ be the closure of the graph of the rational map
$(\frac{1}{1+f_{\eta}},\frac{f_{\eta}}{1+f_{\eta}}):
\overline{\eta}\dasharrow \Delta^1\subset\Ab^2$. This defines the
map of sets $\Gamma:G^{p-1}(X,1)\to Z^p(X,1)$ such that $N\circ
\Gamma$ is the identity. For an element $Y\in Z^p(X,1)$, we have
$\div(N(Y))= d(Y)\in Z^p(X)$, where $d$ denotes the differential in
the complex $Z^p(X,\bullet)$.

Given an equidimensional subvariety $S\subset X$, it is easy to check that
$\Gamma(G^{p-1}_S(X,1))\subset Z_S^p(X,1)$ and
$N(Z_S^p(X,1))\subset G^{p-1}_S(X,1)$.

\begin{lemma}\label{lemma-movingK1}
Suppose that $X$ is smooth over $k$ and either projective or affine.
Let $\Sc=\{S_1,\ldots,S_r\}$ be a collection of equidimensional
closed subvarieties in $X$, $\{f_{\eta}\}\in G^{p-1}(X,1)$ be a
$K_1$-chain such that the support $|\div(\{f_{\eta}\})|$ meets $S_i$
properly for all $i$, $1\le i\le r$; then there exists a $K_1$-chain
$\{g_{\eta}\}\in G^{p-1}_{\Sc}(X,1)$ such that $\{g_{\eta}\}$ is
$K_2$-equivalent to $\{f_{\eta}\}$.
\end{lemma}
\begin{proof}
Denote by $d$ the differential in the complex $Z^p(X,\bullet)$.
Since $\div(\{f_{\eta}\})\in Z_{\Sc}^p(X,0)$, by Lemma
\ref{lemma-moving}, there exists an element $Y'\in Z_{\Sc}^p(X,1)$
such that $d(Y')=\div(\{f_{\eta}\})= d(\Gamma(\{f_{\eta}\}))$. Again
by Lemma \ref{lemma-moving}, there exists an element $Y''\in
Z^p_{\Sc}(X,1)$ such that $d(Y'')=0$ and
$Y''+Y'-\Gamma(\{f_{\eta}\})=d(\widetilde{Y})$ for some
$\widetilde{Y}\in Z^p(X,2)$.

Recall that $(N\circ d)(Z^p(X,2))\subset\Imm(\Tame)\subset G^{p-1}(X,1)$,
see \cite{MS}, Remark on p. 13 for more details. Therefore, the $K_1$-chain
$\{g_{\eta}\}=N(Y''+Y')\in G^{p-1}_{\Sc}(X,1)$
is $K_2$-equivalent to $\{f_{\eta}\}$.
\end{proof}

\begin{corol}\label{corol-movingK1codim}
In notations from Lemma \ref{lemma-movingK1} let
$Z=\div(\{f_{\eta}\})$ and suppose that $\codim_{Z}(Z\cap S_i)\ge
n_i$ for all $i$, $1\le i\le r$ (in particular, $n_i\le
\codim_X(S_i)$); then there exists a $K_1$-chain $\{g_{\eta}\}\in
G^{p-1}(X,1)$ such that $\{g_{\eta}\}$ is $K_2$-equivalent to
$\{f_{\eta}\}$, $\codim_Y(Y\cap S_i)\ge n_i$ for all $i$, $1\le i\le
r$, where $Y=\Supp(\{g_{\eta}\})$, and for each $\eta\in
Z^{p-1}(X)$, we have $\codim_{\div(g_{\eta})} (\div(g_{\eta})\cap
S_i)\ge n_i$ for all $i$, $1\le i\le r$.
\end{corol}
\begin{proof}
We claim that for each $i$, $1\le i\le r$, there exists an equidimensional
subvariety $S'_i\subset X$ of codimension $n_i$ such that $S'_i\supset S_i$
and $Z$ meets $S'_i$ properly. By Lemma \ref{lemma-movingK1}, this immediately implies
the needed statement.

For each $i$, $1\le i\le r$, we prove the existence of $S'_i$ by induction on $n_i$.
Suppose that $n_i=1$. Then there exists an
effective reduced divisor $H\subset X$ such that $H\supset S_i$ and $H$ meets
$Z$ properly: to construct such divisor we have to choose a closed point
on each irreducible
component of $Z$ outside of $S_i$ and take an arbitrary $H$ that does not contain
any of these points and such that $H\supset S_i$.

Now let us do the induction step from $n_i-1$ to $n_i$. Let $\widetilde{S}_i\subset X$
be an equidimensional subvariety that satisfies the needed condition for $n_i-1$.
For each irreducible component of $\widetilde{S}_i$ choose a closed point on it outside
of $S_i$. Also, for each irreducible component of $\widetilde{S}_i\cap Z$
choose a closed
point on it outside of $S_i$. Thus we get a finite set $T$ of closed points in $X$
outside of $S_i$. Let $H$ be an effective reduced divisor on $X$ such that
$H\supset S_i$
and $H\cap T=\emptyset$; then we put $S'_i=\widetilde{S}_i\cap H$.
\end{proof}

Let $W$ be a codimension $q$ cycle on $X$,
$Y$ be an irreducible subvariety of codimension $d-q$ in $X$ that meets
$|W|$ properly,
and let $f$ be a rational function on $Y$ such that $\div(f)$ does not intersect
with $|W|$. We put $f(Y\cap W)=\prod\limits_{x\in Y\cap |W|}{\rm
Nm}_{k(x)/k}(f^{(Y,W;x)}(x))$,
where $(Y,W;x)$ is the intersection index of $Y$ with the cycle $W$
at a point $x\in Y\cap |W|$.

\begin{lemma}\label{lemma-recipr}
Let $X$ be a smooth projective variety over $k$ and let $p+q=d+1$;
consider cycles $Z\in Z^p(X)$, $W\in Z^q(X)$, and $K_1$-chains
$\{f_{\eta}\}\in G^{p-1}_{|W|}(X,1)$, $\{g_{\xi}\}\in
G^{q-1}_{|Z|}(X,1)$ such that $\div(\{f_{\eta}\})=Z$,
$\div(\{g_{\xi}\})=W$. Then we have
$\prod_{\eta}f_{\eta}(\overline{\eta}\cap
W)=\prod_{\xi}g_{\xi}(Z\cap\overline{\xi})$.
\end{lemma}
\begin{proof}
Note that the left hand side depends only in the $K_2$-equivalence class of
the \mbox{$K_1$-chain} $\{f_{\eta}\}$. Therefore by Lemma \ref{lemma-movingK1},
we may assume that $\{f_{\eta}\}\in
G^{p-1}_{\Sc}(X,1)$, where $\Sc=\{\Supp(\{g_{\xi}\}),\cup_{\xi}|\div(g_{\xi})|\}$.
For each pair $\eta\in \Supp(\{f_{\eta}\})$,
$\xi\in\Supp(\{g_{\xi}\})$, let $C^{\alpha}_{\eta\xi}$ be an irreducible component of
the intersection $\overline{\eta}\cap\overline{\xi}$ and
let $n^{\alpha}_{\eta\xi}$ be the intersection index of the subvarieties
$\overline{\eta}$ and $\overline{\xi}$
at the irreducible curve $C^{\alpha}_{\eta\xi}$. By condition, for all
$\eta,\xi,\alpha$ as above,
the restrictions $f^{\alpha}_{\eta\xi}=f_{\eta}|_{C^{\alpha}_{\eta\xi}}$ and
$g^{\alpha}_{\eta\xi}=g_{\xi}|_{C^{\alpha}_{\eta\xi}}$ are well defined
as rational functions on the irreducible curve $C^{\alpha}_{\eta\xi}$. It follows that
$\prod_{\eta}f_{\eta}(\overline{\eta}\cap W)=\prod_{\eta,\xi,\alpha}
f^{\alpha}_{\eta\xi}(\div(g^{\alpha}_{\eta\xi}))^{n^{\alpha}_{\eta\xi}}$;
thus we conclude by the classical
Weil reciprocity law for curves.
\end{proof}

\begin{rmk}
The same reasoning as in the proof of Lemma \ref{lemma-recipr}
is explained in a slightly different language in the
proof of Proposition 3 from \cite{Blo89}.
\end{rmk}

\begin{lemma}\label{lemma-Bloch}
Let $X$ be a smooth projective variety over $k$. Suppose that a
cycle $W\in Z^q(X)$ is homologically trivial; then for any
$K_1$-chain $\{f_{\eta}\}\in G^{d-q}_{|W|}(X,1)$ with
$\div(\{f_{\eta}\})=0$, we have
$\prod_{\eta}f_{\eta}(\overline{\eta}\cap W)=1$.
\end{lemma}
\begin{proof}
By condition, $\Gamma(\{f_{\eta}\})\in Z_{|W|}^{d-q}(X,1)$.
Keeping in mind the explicit formula for product in higher Chow groups
for cycles in general position (see \cite{Blo86}), we see that
\mbox{$((\pi_*)\circ m)(\Gamma(\{f_{\eta}\})\otimes W)\in k^*$}
is well defined and coincides with $\prod_{\eta}f_{\eta}(\overline{\eta}\cap W)$.
Hence we conclude by Lemma \ref{lemma-FirstBloch}.
\end{proof}

\begin{rmk}
If $q=d$, then Lemma \ref{lemma-Bloch} is trivial. An elementary
proof of Lemma \ref{lemma-Bloch} for the case $q=1$ can be found
in~\cite{Gor07}.
\end{rmk}

\subsection{Facts on the Abel--Jacobi map}\label{subsection-Abel-Jacobi}

Notions and results of this section are used in
Section~\ref{subsection-AJbiext}.

Let $X$ be a complex smooth variety of dimension $d$. By $A^n_X$
denote the group of complex valued smooth differential forms on $X$
of degree $n$. Let $F^pA^n_X$ be the subgroup in $A^n_X$ that
consists of all differential forms with at least $p$ ``$dz_i$''. If
$X$ is projective, then the classical Hodge theory implies
$H^n(F^pA^{\bullet}_X)=F^pH^n(X,\C)$, where we consider the Hodge
filtration in the right hand side.

Let $S$ be a closed subvariety in a smooth projective variety $X$;
then the notation $\eta\in F_{\log}^pA^n_{X\backslash S}$ means that
there exists a smooth projective variety $X'$ together with a
birational morphism $f:X'\to X$ such that $D=f^{-1}(S)$ is a normal
crossing divisor on $X'$, $f$ induces an isomorphism $X'\backslash
f^{-1}(S)\to X\backslash S$, and $f^*\eta\in F^pA^n_{X'}\langle
D\rangle$, where $A^n_{X'}\langle D\rangle$ is the group of complex
valued smooth differential forms on $X'\backslash D$ of degree $n$
with logarithmic singularities along $D$. Recall that any class in
$F^pH^n(X\backslash S,\C)$ can be represented by a closed form
$\eta\in F^p_{\log}A^n_{X\backslash S}$ (see \cite{Del71}).

In what follows the variety $X$ is supposed to be projective. By
$H^{*}(X,\Z)$ we often mean the image of this group in
$H^{*}(X,\C)$. Recall that the {\it $p$-th intermediate Jacobian}
$J^{2p-1}(X)$ of $X$ is the compact complex torus given by the
formula
$$
J^{2p-1}(X)=H^{2p-1}(X,\C)/(H^{2p-1}(X,\Z)+F^pH^{2p-1}(X,\C))=
$$
$$
=F^{d-p+1}H^{2d-2p+1}(X,\C)^*/H_{2d-2p+1}(X,\Z).
$$
Let $Z$ be a homologically trivial algebraic cycle on $X$ of
codimension $p$. Then there exists a differentiable singular chain
$\Gamma$ of dimension $2d-2p+1$ with $\partial\Gamma=Z$, where
$\partial$ denotes the differential in the complex of singular
chains on $X$. Consider a closed differential form $\omega\in
F^{d-p+1}A^{2d-2p+1}_X$. It can be easily checked that the integral
$\int_{\Gamma}\omega$ depends only on the cohomology class
$[\omega]\in
F^{d-p+1}H^{2d-2p+1}(X,\C)=H^{2d-2p+1}(F^{d-p+1}A^{\bullet}_X)$ of
$\omega$. Thus the assignment
$$
Z\mapsto\mbox{$\{[\omega]\mapsto \int_{\Gamma}\omega\}$}
$$
defines a homomorphism $AJ:Z^p(X)_{\hom}\to J^{2p-1}(X)$, which is
called the {\it Abel--Jacobi map}.

We give a slightly different description of the Abel--Jacobi map.
Recall that there is an exact sequence of integral mixed Hodge
structures:
$$
0\to H^{2p-1}(X)(p)\to H^{2p-1}(X\backslash
|Z|)(p)\stackrel{\partial_Z}\to H_{2d-2p}(|Z|)\to H^{2p}(X)(p).
$$
Recall that $H_{2d-2p}(|Z|)=\oplus_i\Z(0)$, where the sum is taken
over all irreducible components in $|Z|=\cup_i Z_i$. Thus the cycle
$Z$ defines an element $[Z]\in F^0H_{2d-2p}(|Z|,\C)$ with a trivial
image in the group $H^{2p}(X,\Z(p))$. Hence there exists a closed
differential form $\eta\in F_{\log}^pA^{2p-1}_{X\backslash|Z|}$ such
that $\partial_Z([(2\pi i)^p\eta])=[Z]$. The difference
$PD[\Gamma]-[\eta]$ defines a unique element in the group
$H^{2p-1}(X,\C)$, where $PD:H_{*}(X,|Z|;\Z)\to H^{2d-*}(X\backslash
|Z|,\Z)$ is the canonical isomorphism induced by Poincar\'e duality.

\begin{lemma}\label{lemma-AJexplic}
The image of $PD[\Gamma]-[\eta]\in H^{2p-1}(X,\C)$ in the
intermediate Jacobian $J^{2p-1}(X)$ is equal to the image of $Z$
under the Abel--Jacobi map.
\end{lemma}
\begin{proof}
Consider a closed differential form $\omega\in
F^{d-p+1}A^{2d-2p+1}_X$. Since $\dim(Z)=d-p$, the form $\omega$ also
defines the class $[\omega]\in F^{d-p+1}H^{2d-2p+1}(X,|Z|;\C)$.
Denote by
$$
(\cdot,\cdot): H^{*}(X\backslash Z,\C)\times H^{2d-*}(X,|Z|;\C)\to
\C
$$
the natural pairing. Then we have $(PD[\Gamma]-[\eta],[\omega])=
(PD[\Gamma],[\omega])=\int_{\Gamma}\omega$; this proves the needed
result.
\end{proof}

\begin{rmk}\label{rmk-trivial-AJ}
It follows from Lemma \ref{lemma-AJexplic} that $AJ(Z)=0$ if and
only if there exists an element $\alpha\in
F^pH^{2p-1}(X\backslash|Z|,\C)\cap H^{2p-1}(X\backslash |Z|,\Z(p))$
such that $\partial_Z(\alpha)=[Z]$.
\end{rmk}

\begin{examp}\label{examp-P1}
Suppose that $X=\P^1$, $Z=\{0\}-\{\infty\}$. Let $z$ be a coordinate
on $\P^1$ and $\Gamma$ be a smooth generic path on $\P^1$ such that
$\partial\Gamma=\{0\}-\{\infty\}$; then we have
$\alpha=[\frac{dz}{z}]= 2\pi iPD[\Gamma]\in F^1H^1(X\backslash
|Z|,\C)\cap H^1(X\backslash |Z|,\Z(1))$ and $\partial_Z(\alpha)=Z$.
\end{examp}

\begin{lemma}\label{lemma-AJrational}
Suppose that the cycle $Z$ is rationally trivial; then $AJ(Z)=0$.
\end{lemma}
\begin{proof}[\it Proof 1]

By linearity, it is enough to consider the case when $Z=\div(f)$,
$f\in \C(Y)^*$, $Y\subset X$ is an irreducible subvariety of
codimension $p-1$. Let $\widetilde{Y}$ be the closure of the graph
of the rational function $f:Y\dasharrow\P^1$ and let
$p:\widetilde{Y}\to X$ be the natural map. In \cite{Gri} it was
shown that the following map is holomorphic:
$$
\varphi:\P^1\to J^{2p-1}(X),z\mapsto
AJ(p_*f^{-1}(\{z\}-\{\infty\}));
$$
therefore $\varphi$ is constant and $AJ(Z)=\varphi(0)=\{0\}$.

{\it Proof 2.} Let $\alpha$ be as in Example \ref{examp-P1}; then by
Lemma \ref{lemma-correspAJ} with $X_1=\P^1$, $X_2=X$,
$C=\widetilde{Y}$, $Z_1=\{0,\infty\}$, $W_2=\emptyset$, we have
$[\widetilde{Y}]^*\alpha\in F^pH^{2p-1}(X\backslash |Z|,\C)\cap
H^{2p-1}(X\backslash|Z|,\Z(p))$ and
$\partial_Z([\widetilde{Y}]^*\alpha)=Z$; thus we conclude by Remark
\ref{rmk-trivial-AJ}.
\end{proof}

\begin{rmk}\label{rmk-explicit-chain}
If there is a differentiable triangulation of the closed subset
$p(f^{-1}(\gamma))\subset X$, then we have a well defined class
$[p(f^{-1}(\gamma))]\in H_{2d-2p+1}(X,|Z|;\Z)$ and $\alpha=(2\pi
i)^pPD[p(f^{-1}(\gamma))]$.
\end{rmk}

In particular, we see that the Abel--Jacobi map factors through Chow
groups. In what follows we consider the induced map
$AJ:CH^p(X)_{\hom}\to J^{2p-1}(X)$.

In the second proof of Lemma \ref{lemma-AJrational} we have used the
following simple fact. Let $X_1$ and $X_2$ be two complex smooth
projective varieties of dimensions $d_1$ and $d_2$, respectively.
Suppose that $C\subset X_1\times X_2$, $Z_1\subset X_1$, and
$W_2\subset Z_2$ are closed subvarieties; we put
$Z_2=\pi_2(\pi_1^{-1}(Z_1)\cap C)$, $W_1=\pi_1(\pi_2^{-1}(W_2)\cap
C)$, where $\pi_i:X_1\times X_2\to X_i$, $i=1,2$ denote the natural
projections.

\begin{lemma}\label{lemma-correspAJ}
\begin{itemize}
\item[(i)]
Let $c$ be the codimension of $C$ in $X_1\times X_2$; then there is
a natural morphism of integral mixed Hodge structures
$$
[C]^*:H^{*}(X_1\backslash Z_1,W_1)\to H^{*+2c-2d_1}(X_2\backslash
Z_2,W_2)(c-d_1).
$$
\item[(ii)]
For $i=1,2$, let $p_i$ be the codimension of $Z_i$ in $X_i$. Suppose
that $C$ meets $\pi_1^{-1}(Z_1)$ properly, the intersection
$\pi_1^{-1}(Z_1)\cap\pi_2^{-1}(W_2)\cap C$ is empty, and
$p_1+c-d_1=p_2$; then $Z_1\cap W_1=Z_2\cap W_2=\emptyset$ and the
following diagram commutes:
$$
\begin{array}{ccc}
H^{2p_1-1}(X_1\backslash Z_1,W_1)(p_1)&\stackrel{\partial_{Z_1}}\to&
H_{2d_1-2p_1}(Z_1)
=Z^0(Z_1)\otimes\Z(0)\\
\downarrow\lefteqn{[C]^*}&&\downarrow\lefteqn{\pi_2(C\cap\pi_1^{-1}(\cdot))}\\
H^{2p_2-1}(X_2\backslash Z_2,W_2)(p_2)&\stackrel{\partial_{Z_2}}\to&
H_{2d_2-2p_2}(Z_2)
=Z^0(Z_2)\otimes\Z(0),\\
\end{array}
$$
where the right vertical arrow is defined via the corresponding
natural homomorphisms of groups of algebraic cycles.
\end{itemize}
\end{lemma}
\begin{proof}
The needed morphism $[C]^*$ is the composition of the following
natural morphisms
$$
H^{*}(X_1\backslash Z_1,W_1)\to H^{*}((X_1\times X_2) \backslash
\pi_1^{-1}(Z_1),\pi_2^{-1}(W_2)\cap
C)\stackrel{\cap[C]}\longrightarrow
$$
$$
\stackrel{\cap[C]}\longrightarrow H^{*+2c}_{C\backslash
\pi_1^{-1}(Z_1)}((X_1\times X_2)\backslash
\pi_1^{-1}(Z_1),\pi_2^{-1}(W_2)\cap C)(c)=
$$
$$
=H_{2d_1+2d_2-*-2c} (C\backslash \pi_2^{-1}(W_2),\pi_1^{-1}(Z_1)\cap
C)=
$$
$$
=H^{*+2c}_{C\backslash \pi_1^{-1}(Z_1)}((X_1\times X_2)\backslash
(\pi_1^{-1}(Z_1)\cap C),\pi_2^{-1}(W_2))(c)\to
$$
$$
\to H^{*+2c}((X_1\times X_2)\backslash (\pi_1^{-1}(Z_1)\cap
Y),\pi_2^{-1}(W_2))(c)\to H^{*+2c-2d_1}(X_2\backslash
Z_2,W_2)(c-d_1),
$$
where the first morphism is the natural pull-back map, the second
one is multiplication by the fundamental class $[C]\in
H^{2c}_C(X_1\times X_2)(c)$, the equalities in the middle follow
from the excision property, and the last morphism is the
push-forward map. The second assertion follows from the
commutativity of the following diagram:
$$
\begin{array}{ccc}
H^{*}((X_1\times X_2)\backslash (\pi_1^{-1}(Z_1)\cap
C),\pi_2^{-1}(W_2))&\to&
H_{2d_1+2d_2-1-*}(\pi_1^{-1}(Z_1)\cap C)(-c-p_1)\\
\downarrow&&\downarrow\\
H^{*-2d_1}(X_2\backslash Z_2,W_2)(-d_1)&\to& H_{2d_1+2d_2-1-*}(Z_2)(-c-p_1).\\
\end{array}
$$
\end{proof}

\subsection{Facts on $K$-adeles}\label{subsection-adelic}

Notions and results of this section are used in
Section~\ref{subsection-Kcohombiext}.

Let $X$ be an equidimensional variety over the ground field $k$.
Let $\K_n$ be the sheaf on $X$ associated to the presheaf given by the formula
$U\mapsto K_n(U)$, where $K_n(-)$ is the Quillen $K$-group and $U$
is an open subset in $X$. Zariski cohomology groups of the sheaves $\K_n$ are called
{\it $K$-cohomology groups}. When it will be necessary for us to point
out the underline variety, we will use notation $\K_n^X$ for
the defined above sheaf $\K_n$ on $X$.

Recall that in notations from Section~\ref{subsection-K1},
for all integers $n\ge 1,p\ge 0$, there are natural
homomorphisms \mbox{$d:G^{p}(X,n)\to G^{p+1}(X,n-1)$} such that $d^2=0$. Thus
for each $n\ge 0$, there is a complex $Gers(X,n)^{\bullet}$,
where $Gers(X,n)^p=G^{p}(X,n-p)$; this
complex is called the {\it Gersten complex}. Note that
the homomorphisms $d:G^{p-1}(X,1)\to G^{p}(X,0)$ and \mbox{$d:G^{p-2}(X,2)\to
G^{p-1}(X,1)$} coincide with the homomorphisms $\div$ and $\Tame$, respectively.
For a projective morphism of varieties $f:X\to Y$,
there is a push-forward morphism of complexes
$f_*:Gers(X,n)^{\bullet}\to Gers(Y,n+\dim(Y)-\dim(X))^{\bullet}[\dim(Y)-\dim(X)]$.

In what follows we suppose that $X$ is {\it smooth} over the field
$k$. By results of Quillen (see \cite{Q}), for each $n\ge 0$, there
is a canonical isomorphism between the classes of the complexes
$Gers(X,n)^{\bullet}$ and $R\Gamma(X,\K_n)$ in the derived category
$D^{b}(\Abc)$. In particular, there is a canonical isomorphism
$H^p(X,\K_p)\cong CH^p(X)$ for all $p\ge 0$.
The last statement is often called the Bloch--Quillen formula.

There is a canonical product between the sheaves of $K$-groups, induced by the
product in $K$-groups themselves. However,
the Gersten complex {\it is not multiplicative}, i.e., there is no a
product between Gersten complexes that would correspond to the product between
sheaves $\K_n$: otherwise there would
exist an intersection theory for algebraic cycles without taking
them modulo rational equivalence. Explicitly, there is no a morphism of complexes
$$
Gers(X,m)^{\bullet}\otimes_{\Z} Gers(X,n)^{\bullet}\to
Gers(X,m+n)^{\bullet}
$$
that would correspond to the natural product between cohomology groups
$$
H^{\bullet}(X,\K_m)\otimes H^{\bullet}(X,\K_n)\to
H^{\bullet}(X,\K_{m+n}).
$$
Therefore if one would like to work explicitly with the pairing of
objects in the derived category
$$
R\Gamma(X,\K_m)\otimes^L_{\Z}R\Gamma(X,\K_n)\to R\Gamma(X,\K_{m+n}),
$$
then a natural way would be to use a different resolution rather
than the Gersten complex. There are general multiplicative
resolutions of sheaves, for example a Godement resolution, but it does not see
the Bloch--Quillen isomorphism. In particular, there is no explicit
quasiisomorphism between the Godement and the Gersten complexes.

In \cite{Gor07} the author proposed another way to construct
resolutions for a certain class of abelian sheaves on smooth algebraic
varieties, namely, the {\it adelic resolution}. This class of
sheaves includes the sheaves $\K_n$. It is {\it multiplicative} and there
is an {\it explicit quasiisomorphism} from the adelic resolution to the
Gersten complex.

\begin{rmk}
Analogous adelic resolutions for coherent sheaves on algebraic varieties
have been first introduced by A.\,N.\,Parshin (see \cite{Par76}) in the two-dimensional case,
and then developed by
A.\,A.\,Beilinson (see \cite{Bei}) and A.\,Huber (see \cite{Hub}) in
the higher-dimensional case.
\end{rmk}

\begin{rmk}
When the paper was finished, the author discovered that a similar
but more general construction of a resolution for sheaves on algebraic varieties
was independently done in \cite[Section 4.2.2]{BV} by A.\,A.\,Beilinson and V.\,Vologodsky.
\end{rmk}

Let us briefly recall several notions and facts from \cite{Gor07}. A
{\it non-degenerate flag of length $p$} on $X$ is a sequence of
schematic points $\eta_0\ldots\eta_p$ such that $\eta_{i+1}\in
\overline{\eta}_i$ and $\eta_{i+1}\ne\eta_i$ for all $i$, $0\le i\le
p-1$. For $n,p\ge 0$, there are {\it adelic groups}
$$
\A(X,\K_n)^p\subset\prod_{\eta_0\ldots\eta_p}K_n(\OO_{X,\eta_0}),
$$
where the product is taken over all non-degenerate flags of length
$p$, $\OO_{X,\eta_0}$ is the local ring of the scheme $X$ at a point
$\eta_0$, and the subgroup $\A(X,\K_n)^p$ is defined by certain explicit
conditions concerning ``singularities'' of elements in $K$-groups.
Elements of
the adelic groups are called \mbox{{\it $K$-adeles}} or just adeles. Explicitly, an
adele $f\in\A(X,\K_n)^p$ is a collection
$f=\{f_{\eta_0\ldots\eta_p}\}$ of elements
$f_{\eta_0\ldots\eta_p}\in K_n(\OO_{X,\eta_0})$ that satisfies
certain conditions.

\begin{examp}\label{examp-adelescurve}
If $X$ is a smooth curve over the ground field $k$, then
$$
\A(X,\K_n)^0=K_n(k(X))\times\prod_{x\in X}K_n(\OO_{X,x}),
$$
where the product is taken over all closed points $x\in X$,
and an adele $f\in\A(X,\K_n)^1$ is a collection $f=\{f_{Xx}\}$, $f_{Xx}\in K_n(k(X))$,
such that $f_{Xx}\in K_n(\OO_{X,x})$
for almost all $x\in X$ (this is the restricted product condition in
this case). The apparent similarity with classical adele and idele groups
explains the name of the notion.
\end{examp}

There is a differential
$$
d:\A(X,\K_n)^p\to\A(X,\K_n)^{p+1}
$$
defined by the formula
$$
(df)_{\eta_0\ldots\eta_p}=\sum_{i=0}^p(-1)^i
f_{\eta_0\ldots\hat\eta_i\ldots\eta_p},
$$
where the hat over $\eta_i$ means that we omit a point $\eta_i$. It
can be easily seen that $d^2=0$, so one gets an {\it adelic complex}
$\A(X,\K_n)^{\bullet}$. There is a canonical morphism of complexes
$\nu_X:\A(X,\K_n)^{\bullet}\to Gers(X,n)^{\bullet}$. There is also
an adelic complex of flabby sheaves
$\underline{\A}(X,\K_n)^{\bullet}$ given by the formula
$\underline{\A}(X,\K_n)^{p}(U)=\A(U,\K_n)^p$ and a natural morphism
of complexes of sheaves $\K_n[0]\to
\underline{\A}(X,\K_n)^{\bullet}$.

In what follows we suppose that the
ground field $k$ is {\it infinite and perfect}.

\begin{lemma}\label{lemma-adelicquasiis}(\cite[Theorem 3.34]{Gor07})
The complex of sheaves $\underline{\A}(X,\K_n)^{\bullet}$ is a
flabby resolution for the sheaf $\K_n$. The morphism $\nu_X$ is a
quasiisomorphism; in particular, this induces a canonical
isomorphism between the classes of the complexes
$\A(X,\K_n)^{\bullet}$ and $R\Gamma(X,\K_n)$ in the derived category
$D^{b}(\Abc)$.
\end{lemma}

In particular, there is a canonical isomorphism
$$
H^p(X,\K_n)=H^{p}(\A(X,\K_n)^{\bullet}).
$$

The main advantages of adelic complexes are
the contravariancy and the multiplicativity properties.

\begin{lemma}(\cite[Remark 2.12]{Gor07})
\hspace{0cm}
\begin{itemize}
\item[(i)]
Given a morphism $f:X\to Y$ of smooth varieties over $k$, for each $n\ge 0$,
there is a morphism of complexes
$$
f^*:\A(Y,\K_n^Y)^{\bullet}\to
\A(X,\K_n^X)^{\bullet};
$$
this morphism agrees with the natural morphism
$\Hom_{D^b(\Abc)}(R\Gamma(Y,f_*\K_n^X),R\Gamma(X,\K_n^X))$ and the morphism
$\K_n^Y\to f_*\K_n^X$ of sheaves on $Y$.
\item[(ii)]
For all $p,q\ge 0$, there is a morphism of complexes
$$
m:\A(X,\K_p)^{\bullet}\otimes\A(X,\K_q)^{\bullet}\to
\A(X,\K_{p+q})^{\bullet};
$$
this morphism agrees with the multiplication morphism
$$
m\in \Hom_{D^{b}(\Abc)}(R\Gamma(X,\K_p)\otimes_{\Z}^L R\Gamma(X,\K_q),
R\Gamma(X,\K_{p+q})).
$$
\end{itemize}
\end{lemma}

In what follows we recollect some technical notions and facts that
are used in calculations with elements of the adelic
complex. The idea is to associate with each cocycle in the Gersten complex,
an explicit cocycle in the adelic resolution with the same class in $K$-cohomology.
The adelic cocycle should be good enough so that it would be easy to calculate
its product with other (good) adelic cocycles. This allows to analyze explicitly
the interrelation between the product on complexes
$R\Gamma(X,\K_n)$ and the Bloch--Quillen isomorphism.

For any equidimensional subvariety $Z\subset X$
of codimension $p$ in $X$, there is a notion of a {\it
patching system} $\{Z^{1,2}_r\}$, $1\le r\le p-1$ for $Z$ on $X$, where
$Z^1_r$ and $Z^2_r$ are equidimensional subvarieties in $X$ of codimension $r$
such that the system $\{Z^{1,2}_r\}$ satisfies certain properties;
in particular, we have:
\begin{itemize}
\item[(i)]
the varieties $Z_{r}^1$ and $Z_r^2$ have no common irreducible
components for all $r,1\le r\le p-1$;
\item[(ii)]
the variety $Z$ is contained in both varieties $Z^1_{p-1}$ and $Z_{p-1}^2$,
and the variety $Z_{r}^1\cup Z_r^2$ is contained in both varieties $Z^1_{r-1}$ and
$Z^2_{r-1}$ for all $r,2\le r\le p-1$.
\end{itemize}

\begin{rmk}
What we call here a patching system is what is called in
\cite{Gor07} {\it a patching system with the freedom degree at least
zero}.
\end{rmk}

\begin{lemma}\label{lemma-patchingsyst}(\cite[Remark 3.32]{Gor07})
\hspace{0cm}
\begin{itemize}
\item[(i)]
Suppose that $Z\subset X$ is an equidimensional subvariety of codimension $p$ in $X$;
then there exists a patching system $\{Z^{1,2}_r\}$, $1\le r\le p-1$
for $Z$ on $X$ such that each irreducible component of $Z^1_{p-1}$ and $Z^2_{p-1}$
contains some irreducible component of $Z$ and
for any $r$, $1\le r\le p-2$, each irreducible component of $Z^1_r$ and $Z^2_r$
contains some irreducible component of $Z^1_{r+1}\cup Z^2_{r+1}$;
\item[(ii)]
given an equidimensional subvariety $W\subset X$ that meets $Z$ properly,
one can require in addition that
no irreducible component of $W\cap Z^1_r$ is contained in $Z^2_r$ for all $r$,
$1\le r\le p-1$.
\end{itemize}
\end{lemma}

\begin{rmk}\label{rmk-adelicproper}
If an equidimensional subvariety $W\subset X$ meets $Z$ properly
and the patching system
$\{Z^{1,2}_r\}$ satisfies the condition $(i)$ from Lemma \ref{lemma-patchingsyst},
then $W$ meets $Z^i_r$ properly for $i=1,2$ and all $r$, $1\le r\le p-1$.
\end{rmk}

Suppose that $\{f_{\eta}\}\in Gers(X,n)^p$ is a cocycle in the
Gersten complex and let $Z$ be the support of $\{f_{\eta}\}$. Given
a patching system $Z^{1,2}_r$, $1\le r\le p-1$ for $Z$ on $X$, there
is a notion of a {\it good cocycle}
$[\{f_{\eta}\}]\in\A(X,\K_n)^{p}$ with respect to the patching
system $\{Z^{1,2}_r\}$. In particular, we have $d[\{f_{\eta}\}]=0$,
$\nu_X[\{f_{\eta}\}]=\{f_{\eta}\}$, and
$i_U^*[\{f_{\eta}\}]=0\in\A(U,\K_n)^p$, where $i_U:U=X\backslash
Z\hookrightarrow X$ is the open embedding. Thus the good cocycle
$[\{f_{\eta}\}]$ is a cocycle in the adelic complex
$\A(X,\K_n)^{\bullet}$, which represents the cohomology class in
$H^p(X,\K_n)$ of the Gersten cocycle $\{f_{\eta}\}$. In addition,
$[\{f_{\eta}\}]$ satisfies certain properties that allow to consider
explicitly its products with other cocycles.

\begin{lemma}\label{lemma-adelicgoodcocycles}(\cite[Claim 3.47]{Gor07})
Let $\{f_{\eta}\}$, $Z$, and $\{Z^{1,2}_r\}$, $1\le r\le p-1$ be as
above; then there exists a good cocycle for $\{f_{\eta}\}$ with
respect to the patching system $Z^{1,2}_r$.
\end{lemma}

Consider a cycle $Z\in Z^p(X)$; suppose that a $K_1$-chain
$\{f_{\eta}\}\in G^{p-1}(X,1)$ is such that $\div(\{f_{\eta}\})=Z$ and
a $K$-adele $[Z]\in \A(X,\K_p)^p$ is such that $d[Z]=0$, $\nu_X([Z])=Z$,
and $i_U^*[Z]=0\in\A(U,\K_p)^p$,
where $i_U:U=X\backslash|Z|\hookrightarrow X$ is the open embedding.

\begin{lemma}\label{lemma-adelicboundary}(\cite[Lemma 3.48]{Gor07})
In the above notations, let $Y=\Supp(\{f_{\eta}\})$ and let $\{Y^{1,2}_r\}$
be a patching system for $Y$ on $X$; then there exists a $K$-adele
$[\{f_{\eta}\}]\in\A(X,\K_{p})^{p-1}$ such that
$d[\{f_{\eta}\}]=[Z]$, $\nu_X[\{f_{\eta}\}]=\{f_{\eta}\}$, and
$i_U^*[\{f_{\eta}\}]\in\A(U,\K_p)^{p-1}$ is a good cocycle with
respect to the restriction of the patching system $\{Y^{1,2}_r\}$ to
$U$.
\end{lemma}

For a cycle $Z\in Z^p(X)$, let $\{Z^{1,2}_r\}$, $1\le r\le p-1$ be a
patching system for $|Z|$ on $X$ and $[Z]\in\A(X,\K_p)^p$ be a good
cocycle for $Z$ with respect to the patching system $\{Z^{1,2}_r\}$.
Given a $K_1$-chain $\{f_z\}\in G^{p}(X,1)$ with support on $|Z|$
and $\div(\{f_z\})=0$, let $[\{f_z\}]\in\A(X,p+1)^{p}$ be a good
cocycle for $\{f_z\}$ with respect to the patching system
$\{Z^{1,2}_r\}$. Let $\{W^{1,2}_s\}$, $1\le s\le q-1$, $[W]$, and
$[\{g_w\}]\in\A(X,q+1)^{q}$ be the analogous objects for a cycle
$W\in Z^q(X)$ and a $K_1$-chain $\{g_w\}\in G^{q}(X,1)$ with support
on $|W|$ and $\div(\{g_w\})=0$.

\begin{lemma}\label{lemma-adelicproduct}(\cite[Theorem 4.22]{Gor07})
In the above notations, suppose that $p+q=d$, $|Z|$ meets $|W|$
properly, the patching systems $\{Z^{1,2}_r\}$ and $\{W^{1,2}_s\}$
satisfy the condition $(i)$ from Lemma \ref{lemma-patchingsyst}, and
that the patching system $\{W^{1,2}_s\}$ satisfies the condition
$(ii)$ from Lemma \ref{lemma-patchingsyst} with respect to the
subvariety $|Z|$. Then we have
$$
\nu_X(m([\{f_z\}]\otimes[W]))=(-1)^{(p+1)q}\{(\prod_{z\in Z^{(0)}}f_z^{(\overline{z},W;x)}(x))_x\}\in G^d(X,1),
$$
$$
\nu_X(m([Z]\otimes[\{g_w\}]))=(-1)^{pq}\{(\prod_{w\in W^{(0)}}g_w^{(Z,\overline{w};x)}(x))_x\}\in G^d(X,1),
$$
where $(\overline{z},W;x)$ is the intersection index of the
subvariety $\overline{z}$ with the cycle $W$ at a point
$x\in X^{(d)}$ (the same for $(Z,\overline{w};x)$).
\end{lemma}

\subsection{Facts on determinant of cohomology and Picard categories}\label{subsection-determ}

Notions and results of this section are used in
Section~\ref{subsection-determbiext}.

Let $X$ be a smooth projective variety over a field $k$. Given two
coherent sheaves $\F$ and $\G$ on $X$, one has a well-defined
$k^*$-torsor $\det
R\Gamma(X,\F\otimes^{L}_{\OO_X}\G)\backslash\{0\}$, where
$\det(V^{\bullet})=\otimes_{i\in\Z}\det^{(-1)^{i}}H^i(V^{\bullet})$
for a bounded complex of vector spaces $V^{\bullet}$. We will need
to study a behavior of this $k^*$-torsor with respect to exact
sequences of coherent sheaves. With this aim it is more convenient
to consider a $\Z$-graded $k^*$-torsor
$$
\langle\F,\G\rangle=(\rk R\Gamma(X,\F\otimes^{L}_{\OO_X}\G),\det R\Gamma(X,\F\otimes^{L}_{\OO_X}\G)\backslash\{0\})
$$
and to use the construction of virtual coherent sheaves and virtual vector spaces.

Recall that for any exact category $\Cc$, P.\,Deligne has defined in
\cite{Del} the category of virtual objects $V\Cc$ together with a
functor $\gamma:\Cc_{iso}\to V\Cc$ that has a certain universal
property, where $\Cc_{iso}$ is the category with the same objects as
$\Cc$ and with morphisms being all isomorphisms in $\Cc$. The
category $V\Cc$ is a {\it Picard category}: it is non-empty, every morphism in
$V\Cc$ is invertible, there is a functor $+:V\Cc\times V\Cc\to V\Cc$
such that it satisfies some compatible associativity and
commutativity constraints and such that for any object $L$ in
$V\Cc$, the functor $(\cdot+L):V\Cc\to V\Cc$ is an autoequivalence
of $V\Cc$ (see op.cit.)\footnote{In op.cit. this notion is called
a {\it commutative} Picard category; since we do not consider
non-commutative Picard categories,
we use a shorter terminology.}. In particular, this implies the existence of a unit object
$0$ in any Picard category.

Explicitly, the objects of $V\Cc$ are based loops on the
$H$-space $BQ\Cc$ and for any two loops $\gamma_1$, $\gamma_2$ on
$BQ\Cc$, the morphisms in $\Hom_{V\Cc}(\gamma_1,\gamma_2)$ are the
homotopy classes of homotopies from $\gamma_1$ to $\gamma_2$. For an
object $E$ in $\Cc$, the object $\gamma(E)$ in $V\Cc$ is the
canonical based loop on $BQ\Cc$ associated with $E$. By $[E]$
denote the class of $E$ in the group $K_0(\Cc)$. Note that
$[E]=[E']$ if and only if there is a morphism between $\gamma(E)$
and $\gamma(E')$ in $V\Cc$. Moreover, an exact sequence in $\Cc$
$$
0\to E'\to E\to E''\to 0
$$
defines in a canonical way an isomorphism in $V\Cc$
$$
\gamma(E')+\gamma(E'')\cong \gamma(E).
$$
Furthermore, there are canonical isomorphisms
$\pi_{i}(V\Cc)\cong K_i(\Cc)$ for $i=0,1$ and $\pi_{i}(V\Cc)=0$ for
$i>1$.

Let us explain in which sense the functor $\gamma:\Cc_{iso}\to V\Cc$ is universal.
Given a Picard category $\Pc$, consider the category $\mathrm{Det}(\Cc,\Pc)$
of {\it determinant functors}, i.e.,
a category of pairs $(\delta,D)$, where $\delta:\Cc_{iso}\to\Pc$
is a functor and a $D$ is a functorial isomorphism
$$
D(\Sigma):\delta(E')+\delta(E'')\to \delta(E)
$$
for each exact sequence
$$
\Sigma:0\to E'\to E\to E''\to 0
$$
such that the pair $(\delta,D)$ is compatible with zero objects,
associativity and commutativity (see op.cit., 4.3). Note that $\mathrm{Det}(\Cc,\Pc)$
has a natural structure of a Picard category.

For Picard categories $\Pc$ and $\Qc$, denote by $\Funt(\Pc,\Qc)$ the category
of symmetric monoidal functors $F:\Pc\to\Qc$. Morphisms in $\Funt(\Pc,\Qc)$
are monoidal morphisms between monoidal functors. Denote by $0:\Pc\to \Qc$ a
functor that sends every object of $\Pc$ to the unit object $0$ in $\Qc$
and sends every morphism to the identity. Note that $\Funt(\Pc,\Qc)$ has a natural
structure of a Picard category.
The universality of $V\Cc$ is expressed by the following statement (see op.cit., 4.4).

\begin{lemma}\label{lemma-univ}
For any Picard category $\Pc$ and an exact category $\Cc$,
the composition with the functor $\gamma:\Cc\to V\Cc$ defines an equivalence
of Picard categories
$$
\Funt(V\Cc,\Pc)\to\mathrm{Det}(\Cc,\Pc).
$$
\end{lemma}

By universality of $V\Cc$, the smallest
Picard subcategory in $V\Cc$ containing $\gamma(\Cc_{iso})$ is
equivalent to the whole category $V\Cc$.

\begin{examp}\label{examp-virtualfield}
Let us describe explicitly the category $V\M_k$, where $\M_k$ is the exact
category of finite-dimensional vector spaces over a field $k$
(see op.cit., 4.1 and 4.13). Objects of $V\M_k$ are $\Z$-graded $k^*$-torsors, i.e,
pairs $(l,L)$, where $l\in\Z$ and $L$ is a $k^*$-torsor. For any objects
$(l,L)$ and $(m,M)$ in $V\M_k$, we have $\Hom_{V\M_k}((l,L),(m,M))=0$ if $l\ne
m$, and $\Hom_{V\M_k}((l,L),(m,M))=\Hom_{k^*}(L,M)$ if $l=m$. The
functor $(\M_k)_{iso}\to V\M_k$ sends a vector space $V$ over $k$
to the pair $(\rk_k(V),\det_k(V)\backslash\{0\})$. We have
$$
(l,L)+(m,M)=(l+m,L\otimes M)
$$
and the commutativity constraint
$$
(l,L)+(m,M)\cong (m,M)+(l,L)
$$
is given by the formula $u\otimes v\mapsto (-1)^{lm}v\otimes u$, where $u\in L$, $v\in M$.
\end{examp}

Given exact categories $\Cc$, $\Cc'$, $\Ec$, and a biexact
functor $\Cc\times\Cc'\to\Ec$, one has a corresponding functor
$G:V\Cc\times V\Cc'\to V\Ec$, which is {\it distributive} with respect
to addition in categories of virtual objects: for a fixed object $L$ in $V\Cc$ or $M$ in $V\Cc'$,
the functor $G(L,M)$ is a symmetric monoidal functor and the choices of a
fixed argument are compatible with each other (see op.cit., 4.11). A
distributive functor is an analog for Picard
categories of what is a biextension for abelian groups (see Section~\ref{subsection-quotientbiext}).

Now let $\M_X$, $\Pc_X$, and $\Pc'_X$ denote the categories of
coherent sheaves on $X$, vector bundles on $X$, and vector bundles
$E$ on $X$ such that $H^i(X,E)=0$ for $i>0$, respectively. The
natural symmetric monoidal functors $V\Pc_X\to V\M_X$ and $V\Pc'_X\to V\Pc_X$ are
equivalences of categories, see op.cit., 4.12. A choice of
corresponding finite resolutions for all objects gives inverse
functors to these functors. Thus, there is a distributive functor
$$
\langle\cdot,\cdot\rangle:
V\M_X\times V\M_X\cong V\Pc_X\times V\Pc_X\to V\Pc_X\cong V\Pc'_X\to V\M_k.
$$
There is a canonical isomorphism of
the composition
$\langle\cdot,\cdot\rangle\circ\gamma:(\M_X)_{iso}\times(\M_X)_{iso}\to V\M_k$
and the previously defined functor
$$
\langle\F,\G\rangle=(\rk R\Gamma(X,\F\otimes^{L}_{\OO_X}\G),\det R\Gamma(X,\F\otimes^{L}_{\OO_X}\G)\backslash\{0\})
$$
By Lemma~\ref{lemma-univ}, the last condition defines the distributive functor
$\langle\cdot,\cdot\rangle$ up to a unique isomorphism.

%Note that if $\F_i$ and $\G_j$
%are finite ordered sets of coherent sheaves on $X$ such that
%$\Supp(\F_i)\cap \Supp(\G_j)=\emptyset$ for all $i,j$, then there is a canonical isomorphism
%$\langle\sum_i\gamma(\F_i),\sum_j\gamma(\G_j)\rangle\cong (0,k^*)$.

In Section~\ref{subsection-determbiext} we will need some more generalities about Picard categories that we describe below.
For Picard categories $\Pc_1$ and $\Pc_2$, denote by $\Distr(\Pc_1,\Pc_2;\Qc)$ the category
of distributive functors $G:\Pc_1\times\Pc_2\to\Qc$. Morphisms between $G$ and $G'$
in $\Distr(\Pc_1,\Pc_2;\Qc)$ are morphisms between functors
such that for a fixed object $L$ in $\Pc_1$ or
$M$ in $\Pc_2$, the corresponding morphism of monoidal functors $G(L,M)\to G'(L,M)$
is monoidal. Denote by $0:\Pc_1\times\Pc_2\to \Qc$ a
functor that sends every object of $\Pc_1\times \Pc_2$ to the unit object $0$ in $\Qc$
and sends every morphism to the identity. As above, the category $\Distr(\Pc_1,\Pc_2;\Qc)$
has a natural structure of a Picard category.

We will use quotients of Picard categories.

\begin{defin}\label{defin-localiz}
Given a functor $F$ in $\Funt(\Pc',\Pc)$, a {\it quotient} $\Pc/\Pc'$
is the following category: objects of $\Pc/\Pc'$
are the same as in $\Pc$, and morphisms are defined by the formula
$$
\Hom_{\Pc/\Pc'}(L,M)=\colim_{K\in Ob(\Pc')}\Hom_{\Pc}(L,M+F(K)),
$$
i.e., we take the colimit of the functor $\Pc\to Sets$, $K\mapsto \Hom_{\Pc}(L,M+F(K))$.
%$$
%\Hom_{\Pc/\Pc'}(L,M)=\lim_{Z\in Ob(\Pc')}\Hom_{\Pc}(L,M+F(K))/\pi_1(\Pc').
%$$
The composition of morphisms $f\in\Hom_{\Pc/\Pc'}(L,M)$ and $g\in \Hom_{\Pc/\Pc'}(M,N)$
is defined as follows: represent $f$ and $g$ by elements
$\widetilde{f}\in \Hom_{\Pc}(L,M+F(K_1))$ and
$\widetilde{g}\in \Hom_{\Pc}(M,N+F(K_2))$, respectively, and take the composition
$$
(\widetilde{g}+id_{F(K_1)})\circ\widetilde{f}\in\Hom_{\Pc}(L,N+F(K_1)+F(K_2))=\Hom_{\Pc}(L,N+F(K_1+K_2)).
$$
\end{defin}

\begin{rmk}\label{rmk-localiz}
\hspace{0cm}
\begin{itemize}
%\item[(i)] The group $\pi_1(\Pc')$ acts on the set $\Hom_{\Pc}(L,M+F(K))$
%via the isomorphism $\pi_1(\Pc')\cong \Hom_{\Pc'}(K,K)$,
%the homomorphism $F:\Hom_{\Pc'}(K,K)\to \Hom_{\Pc}(F(K),F(K))$, and the action of
%the latter group on the object $L+F(K)$ via the second summand.

\item[(i)] Taking a representative $K\in Ob(\Pc')$ for each class in $\pi_0(\Pc')$, we get
a decomposition
$$
\Hom_{\Pc/\Pc'}(L,M)=
\coprod_{[K]\in\pi_0(\Pc')}\Hom_{\Pc}(L,M+F(K))/\pi_1(\Pc'),
$$
where $\pi_1(\Pc')$ acts on $\Hom_{\Pc}(L,M+F(K))$ via the second summand
and the identification $\pi_1(\Pc')=\Hom_{\Pc'}(K,K)$.

\item[(ii)] A more natural way to define a quotient of Picard categories
would be to consider 2-Picard categories instead of taking quotients of sets of morphisms
(or, more generally, to consider homotopy quotients of $\infty$-groupoids);
the above construction is a truncation to the first two layers
in the Postnikov tower of the ``right'' quotient.
\end{itemize}
\end{rmk}

Let us consider a descent property for quotients of Picard categories.

For a symmetric monoidal functor $F:\Pc'\to \Pc$ and a Picard category $\Qc$,
consider a category $\Funt_{F}(\Pc,\Qc)$,
whose objects are pairs $(G,\Psi)$, where $G$ is an object
in $\Funt(\Pc,\Qc)$ and $\Psi:G\circ F\to 0$ is an isomorphism in $\Funt(\Pc',\Qc)$.
Morphisms in $\Funt_{F}(\Pc,\Qc)$ are defined in a natural way.
Note that $\Funt_{F}(\Pc,\Qc)$ is a Picard category.

Analogously, for two symmetric monoidal functors of Picard categories $F_1:\Pc'_1\to \Pc_1$,
$F_2:\Pc'_2\to \Pc_2$, and a Picard category $\Qc$, consider a category $\Distr_{(F_1,F_2)}(\Pc_1,\Pc_2;\Qc)$,
whose objects are triples $(G,\Psi_1,\Psi_2)$, where $G$ is an object
in $\Distr(\Pc_1,\Pc_2;\Qc)$, $\Psi_i:G\circ (F_i\times\mathrm{Id})\to 0$, $i=1,2$,
are isomorphisms in $\Distr(\Pc'_i,\Pc_{3-i};\Qc)$ such that
$$
\Psi_1\circ (\mathrm{Id}\times F_{2})=\Psi_2\circ (F_1\times\mathrm{Id})\in
\Hom(G\circ(F_1\times F_2),0),
$$
where morphisms are taken in the category $\Distr(\Pc'_1,\Pc'_2;\Qc)$.
Morphisms in $\Distr_{(F_1,F_2)}(\Pc_1,\Pc_2;\Qc)$ are defined in a natural way.
Note that $\Distr_{(F_1,F_2)}(\Pc_1,\Pc_2;\Qc)$ is a Picard category.

\begin{lemma}\label{lemma-localiz}
\hspace{0cm}
\begin{itemize}
\item[(i)] The category $\Pc/\Pc'$ inherits a canonical Picard category
structure such that the natural functor
$P:\Pc\to \Pc/\Pc'$ is a symmetric monoidal functor.
\item[(ii)]  There is an exact sequence of abelian groups
$$
\pi_1(\Pc')\to \pi_1(\Pc)\to \pi_1(\Pc/\Pc')\to\pi_0(\Pc)\to
\pi_0(\Pc')\to\pi_0(\Pc/\Pc')\to 0.
$$
\item[(iii)] The composition of functors with $P$ defines an isomorphism of Picard categories
$$
\Funt(\Pc/\Pc',\Qc)\to \Funt_{F}(\Pc,\Qc).
$$
\item[(iv)] The composition of functors with $P_1\times P_2$
defines an isomorphism of Picard categories
$$
\Distr(\Pc_1/\Pc'_1,\Pc_2/\Pc_2';\Qc)\to \Funt_{(F_1\times F_2)}(\Pc_1,\Pc_2;\Qc).
$$
\end{itemize}
\end{lemma}
\begin{proof}
(i) Let $L,M$ be objects in $\Pc/\Pc'$; by definition, the sum $L+M$ in $\Pc/\Pc'$
is equal the sum of $L$ and $M$ as objects of the Picard category $\Pc$.
The commutativity and associativity constraints are also induced by that
in the category $\Pc$. The check of the needed properties is straightforward.

(ii) Direct checking that uses an explicit description of morphisms given in
Remark~\ref{rmk-localiz}(i).

(iii) Let $H:\Pc/\Pc'\to \Qc$ be a symmetric monoidal functor and
let $K$ be an object in $\Pc'$. There is a canonical morphism
$F(K)\to 0$ in $\Pc/\Pc'$ given by the morphism $F(K)\to F(K)$ in $\Pc$;
this induces a canonical morphism $(H\circ P\circ F)(K)\to 0$ in $\Qc$
and an isomorphism of symmetric monoidal functors $\Psi:H\circ P\circ F\to 0$.
Thus the assignment $H\mapsto (H\circ P,\Psi)$ defines defines a functor
$\Funt(\Pc/\Pc',\Qc)\to \Funt_{F}(\Pc,\Qc)$.

Let $(G,\Psi)$ be an object in
$\Funt_{F}(\Pc,\Qc)$. We put $H(L)=G(L)$ for any object $L$ in $\Pc/\Pc'$.
Let $f:L\to M$ be a morphism in $\Pc/\Pc'$ represented by a morphism
$\widetilde{f}:L\to M+F(K)$ in $\Pc$. We put $H(f)$ to be the composition
$$
H(L)=G(L)\stackrel{G(\widetilde{f})}\longrightarrow G(M)+(G\circ F)(K)
\stackrel{id_{G(M)}+\Psi(K)}\longrightarrow G(M)=H(M).
$$
This defines a symmetric monoidal functor $H:\Pc/\Pc'\to \Qc$. The assignment
$(G,\Psi)\mapsto H$ defines an inverse functor $\Funt_{F}(\Pc,\Qc)\to \Funt(\Pc/\Pc',\Qc)$
to the functor constructed above.

(iv) The proof is completely analogous to the proof of (iii).
\end{proof}

For an abelian group $A$, denote by $\Picard(A)$ a Picard category whose
objects are elements of $A$ and whose only morphisms are identities.

\begin{examp}\label{examp-localiz}
\hspace{0cm}
\begin{itemize}
\item[(i)] Let $\Pc$ be a Picard category and let $0_{\Pc}$ be a full subcategory in $\Pc$,
whose only object is $0$; then the natural functor $\Pc/0_{\Pc}\to \Picard(\pi_0(\Pc))$
is an equivalence of Picard categories.
\item[(ii)] Let $\Ac'$ be an abelian Serre subcategory in an abelian category $\Ac$;
then the natural functor $V\Ac/V\Ac'\to V(\Ac/\Ac')$ is an equivalence of Picard categories.
This follows immediately from the 5-lemma and Lemma~\ref{lemma-localiz} (ii).
\end{itemize}
\end{examp}

%For an abelian group $N$, denote by $BN$ the
%Picard category with one object $0$ and with $\pi_1(BN)=N$.
%Note that for abelian groups $A$, $B$,
%a biextension of $(A,B)$ by $N$ is the same as a distributive functor
%$\Picard(A)\times\Picard(B)\to BN$.

For an abelian group $N$, denote by $N_{tors}$ the Picard category of $N$-torsors.
Note that for abelian groups $A$, $B$, a biextension of $(A,B)$ by $N$ is
the same as a distributive functor $\Picard(A)\times\Picard(B)\to N_{tors}$.
Let $\Pc_1$, $\Pc_2$ be Picard categories,
$P$ be a biextension of $(\pi_0(\Pc_1),\pi_0(\Pc_2))$ by $N$. Then $P$ naturally defines a
distributive functor $G_P:\Pc_1\times\Pc_2\to N_{tors}$.

%(up to a unique isomorphism) as follows:
%for each pair $(a,b)\in (\pi_0(\Pc_1),\pi_0(\Pc_2))$ choose an isomorphism $P_{(a,b)}\cong N$
%of $N$-torsors; then the action of $G_P$ on morphisms is trivial
%(every morphism maps to a corresponding identity morphism)
%and the distributive structure of $G_P$ is naturally induced
%by the biextension structure on $P$.

Combining Lemma~\ref{lemma-localiz}(iv) and Example~\ref{examp-localiz}(i),
we get the next statement.

\begin{lemma}\label{lemma-Deligne}
A distributive functor $G:\Pc_1\times\Pc_2\to N_{tors}$ is isomorphic to the
functor $G_P$ for some biextension $P$ of
$(\pi_0(\Pc_1),\pi_0(\Pc_2))$ by $N$ if and only if
$G_*(\pi_0(\Pc_1)\times\pi_1(\Pc_2))=G_*(\pi_1(\Pc_1)\times\pi_0(\Pc_2))=0\in
N=\pi_1(N_{tors})$.
\end{lemma}

\section{General facts on biextensions}\label{section-quotientbiext}

\subsection{Quotient biextension}\label{subsection-quotientbiext}

For all groups below, we write the groups law in the additive
manner. For an abelian group $A$ and a natural number $l\in \Z$, let
$A_l$ denote the $l$-torsion in $A$.

The notion of a biextension was first introduced in \cite{Mum}; see
more details on biextensions in \cite{Bre} and \cite{SGA}. Recall
that the set of isomorphism classes of biextensions of $(A,B)$ by $N$ is canonically
bijective with the group $\Ext^1_{\Abc}(A\otimes^L_{\Z}B,N)$ for any
abelian groups $A$, $B$, and $N$, where $\Abc$ is the category of
all abelian groups.

Let us describe one explicit construction of biextensions. As
before, let $A$ and $B$ be abelian groups.

\begin{defin}\label{defin-bigroup}
A subset $T\subset A\times B$ is a {\it bisubgroup} if for all
elements $(a,b)$, $(a,b')$, and $(a',b)$ in $T$, the elements
$(a+a',b)$ and $(a,b+b')$ belong to $T$. For a bisubgroup $T\subset
A\times B$ and an abelian group $N$, a {\it bilinear map} $\psi:T\to
N$ is a map of sets such that for all elements $(a,b)$, $(a,b')$,
and $(a',b)$ in $T$, we have $\psi(a,b)+\psi(a',b)=\psi(a+a',b)$ and
$\psi(a,b)+\psi(a,b')=\psi(a,b+b')$.
\end{defin}

Suppose that $T\subset A\times B$ is a bisubgroup and
$\varphi_A:A\to A'$, $\varphi_B:B\to B'$ are surjective group
homomorphisms such that $(\varphi_A\times\varphi_B)(T)=A'\times B'$.
Consider the set
$$
S=T\cap(\Ker(\varphi_A)\times B\cup A\times\Ker(\varphi_B))\subset
A\times B;
$$
clearly, $S$ is a bisubgroup in $A\times B$. Given a bilinear map
$\psi:S\to N$, let us define an equivalence relation on $N\times T$
as the transitive closure of the isomorphisms
$$
N\times\{(a,b)\}\stackrel{\psi(a',b)}\longrightarrow
N\times\{(a+a',b)\},
$$
$$
N\times\{(a,b)\}\stackrel{\psi(a,b')}\longrightarrow
N\times\{(a,b+b')\}
$$
for all $(a,b)\in T$ and $(a',b),(a,b')\in S$. It is easy to check
that the quotient $P_{\psi}$ of $N\times T$ by this equivalence
relation has a natural structure of a biextension of $(A',B')$ by
$N$.

\begin{rmk}\label{rmk-biextcoinc}
Suppose that we are given two bisubgroups $T_2\subset T_1$ in
$A\times B$ satisfying the above conditions; if the restriction of
$\psi_1$ to $S_2$ is equal to $\psi_2$, then the biextensions
$P_{\psi_1}$ and $P_{\psi_2}$ are canonically isomorphic.
\end{rmk}

\begin{rmk}\label{rmk-isombiext}
Suppose that we are given two bilinear maps $\psi_1,\psi_2:S\to N$.
Let $\phi:T\to N$ be a bilinear map that satisfies
$\psi_1=\phi|_S+\psi_2$; then the multiplication by $\phi$ defines
an isomorphism of the biextensions $P_{\psi_2}$ and $P_{\psi_1}$ of
$(A',B')$ by $N$.
\end{rmk}

\begin{rmk}\label{rmk-quotient-biext}
There is a natural generalization of the construction mentioned
above. Suppose that we are given a biextension $Q$ of $(A,B)$ by $N$
and surjective homomorphisms $\varphi_A:A\to A'$, $\varphi_B:B\to
B'$. Consider the bisubgroup $S=\Ker(\varphi_A)\times B\cup
A\times\Ker(\varphi_B)$ in $A\times B$; a bilinear map (in the
natural sense) $\psi:S\to Q$ is called a {\it trivialization of the
biextension $Q$ over the bisubgroup $S$}. The choice of the
trivialization $\psi:S\to Q$ canonically defines a biextension $P$
of $(A',B')$ by $N$ such that the biextension
$(\varphi_A\times\varphi_B)^*P$ is canonically isomorphic to the
biextension $Q$ (one should use the analogous construction to the
one described above).
\end{rmk}

Now let us recall the definition of a Weil pairing associated to a
biextension.

\begin{defin}
Consider a biextension $P$ of $(A,B)$ by $N$ and a natural number
$l\in{\mathbb N}$, $l\ge 1$; for elements $a\in A_l$, $b\in B_l$,
their {\it Weil pairing} $\phi_l(a,b)\in N_l$ is the obstruction to
a commutativity of the diagram
$$
\begin{array}{ccc}
P_{(a,lb)}&\stackrel{\alpha}\longrightarrow&P^{\wedge l}_{(a,b)}\\
\uparrow\lefteqn{\beta}&&\uparrow\lefteqn{\gamma}\\
P_{(0,0)}&\stackrel{\delta}\longrightarrow& P_{(la,b)},\\
\end{array}
$$
where arrows are canonical isomorphisms of $N$-torsors given by the
biextension structure on $P$: $\phi_l(a,b)=\gamma\circ\delta-\alpha\circ\beta$.
\end{defin}

It is easily checked that the Weil pairing defines a bilinear map
$\phi_l:A_l\times B_l\to N_l$.

\begin{rmk}
There is an equivalent definition of the Weil pairing:
an isomorphism class of a biextension is defined by a morphism
$A\otimes^L_{\Z}B[-1]\to N$ in $D^b(\Abc)$, and the corresponding Weil pairing is obtained
by taking the composition
$$
\oplus_l (A_l\otimes B_l)\to\mathrm{Tor}_1^{\Z}(A,B)\to A\otimes^L_{\Z}B[-1]\to N.
$$
\end{rmk}

\begin{lemma}\label{lemma-Weil-pair-commut}
Let $\varphi_A:A\to A'$, $\varphi_B:B\to B'$, $T\subset A\times B$,
and $\psi:S\to N$ be as above. Consider a pair $(a,b)\in T$ such
that $\varphi_A(la)=0$ and \mbox{$\varphi_B(lb)=0$}; then we have
$$
\phi_l(\varphi_A(a),\varphi_B(b))=\psi(la,b)-\psi(a,lb),
$$
where $\phi_l$ denotes the Weil pairing associated to the
biextension $P$ of $(A',B')$ by $N$ induced by the bilinear map
$\psi:S\to N$.
\end{lemma}
\begin{proof}
By construction the pull-back of the biextension $P$ with respect to
the map $T\to A'\times B'$ is canonically trivial. Furthermore, the
pull-back of the diagram that defines the Weil pairing
$\phi_l(\varphi(a),\varphi(b))$ is equal to the following diagram:
$$
\begin{array}{ccc}
N\times\{(a,lb)\}&\stackrel{id_N}\longrightarrow& N^{\wedge l}\times\{(a,b)\}\\
\uparrow\lefteqn{+\psi(a,lb)}&&\uparrow\lefteqn{id_N}\\
N\times\{(0,0)\}&\stackrel{+\psi(la,b)}\longrightarrow&N\times\{(la,b)\} .\\
\end{array}
$$
This concludes the proof.
\end{proof}

\subsection{Biextensions and pairings between complexes}\label{subsection-complexes}

We describe a way to construct a biextension starting from a pairing
between complexes. Suppose that $A^{\bullet}$ and $B^{\bullet}$ are
two bounded from above complexes of abelian groups, $N$ is an
abelian group, and that we are given a pairing
\mbox{$\phi\in\Hom_{D^-(\Abc)}
(A^{\bullet}\otimes^L_{\Z}B^{\bullet},N)$}. We fix an integer $p$.
Let $H^p(A^{\bullet})'\subset H^p(A^{\bullet})$ be the annulator of
$H^{-p}(B^{\bullet})$ with respect to the induced pairing
$\phi:H^p(A^{\bullet})\times H^{-p}(B^{\bullet})\to N$. Analogously,
we define the subgroup $H^{1-p}(B^{\bullet})'\subset
H^{1-p}(B^{\bullet})$. Let $\tau'_{\le p}A^{\bullet}\subset
A^{\bullet}$ be a subcomplex such that $(\tau'_{\le
p}A^{\bullet})^n=0$ if $n>p$, $(\tau'_{\le p}A^{\bullet})^n=A^n$ if
$n<p$, and $(\tau'_{\le p}A^{\bullet})^p=\Ker(d_A^p)'$, where
$d_A^p:A^p\to A^{p+1}$ is the differential in the complex
$A^{\bullet}$ and $\Ker(d_A^p)'$ is the group of cocycles that map
to $H^p(A^{\bullet})'$. Note that the operation $\tau'_{\le p}$ is
well defined for complexes up to quasiisomorphisms, i.e., the class
of the complex $\tau'_{\le p}A^{\bullet}$ in the derived category
$D^-(\Abc)$ depends only on the classes of the complexes
$A^{\bullet}$, $B^{\bullet}$ and the morphism $\phi$. Analogously,
we define the subcomplex $\tau'_{\le(d+1-p)}B^{\bullet}\subset
B^{\bullet}$. The restriction of $\phi$ defines an element $\phi'\in
\Ext^1_{D^-(\Abc)}(\tau'_{\le p}A^{\bullet}\otimes^L_{\Z}
\tau'_{\le(1-p)}B^{\bullet},N[-1])$. It can be easily checked that
$\phi'$ passes in a unique way through the morphism $\tau'_{\le
p}A^{\bullet}\otimes^L_{\Z} \tau'_{\le(1-p)}B^{\bullet}\to
H^p(A^{\bullet})'[-p]\otimes^L_{\Z}H^{1-p}(B^{\bullet})'[p-1]$. Thus
we get a canonical element
$$
\phi'\in\Hom_{D^-(\Abc)}(H^p(A^{\bullet})'[-p]
\otimes^L_{\Z}H^{1-p}(B^{\bullet})'[1-p],N)=
$$
$$
=\Ext^1_{\Abc}(H^p(A^{\bullet})'\otimes^L_{\Z}H^{1-p}(B^{\bullet})',N),
$$
i.e., a biextension $P_{\phi}$ of
$(H^p(A^{\bullet})',H^{1-p}(B^{\bullet})')$ by $N$.

We construct this biextension explicitly for the case when $\phi$ is
induced by a true morphism of complexes $\phi:A^{\bullet}\otimes
B^{\bullet}\to N$ (this can be always obtained by taking projective
resolutions). Let $A=\Ker(d_A^p)'$, $B=\Ker(d_B^{1-p})'$, $T=A\times
B$, $\phi_A:\Ker(d_A^p)'\to H^p(A^{\bullet})'$,
$\phi_B:\Ker(d_B^{1-p})'\to H^{1-p}(B^{\bullet})'$,
$S=\Imm(d^{p-1}_A)\times
\Ker(d^{1-p}_{B})'\cup\Ker(d_A^p)'\times\Imm(d_B^{-p})$. We define a
bilinear map $\psi:S\to N$ as follows:
$\psi(d_A^{p-1}(a^{p-1}),b^{1-p})=\phi(a^{p-1}\otimes b^{1-p})$ and
$\psi(a^{p},d_B^{-p}(b^{-p}))=(-1)^p\phi(a^{p}\otimes b^{-p})$. It
is readily seen that this does not depend on the choices of
$a^{p-1}$ and $b^{-p}$, and that
$\psi(d^{p-1}_A(a^{p-1}),d_B^{-p}(b^{-p}))$ is well defined (the
reason is that $\phi$ is a morphism of complexes). The application
of the construction from Section~\ref{subsection-quotientbiext}
gives a biextension $P_{\psi}$ of
$(H^p(A^{\bullet})',H^{1-p}(B^{\bullet})')$ by $N$; we have
$P_{\phi}=P_{\psi}$.

\begin{rmk}\label{rmk-Picardbiext}
There is an equivalent construction of the biextension $P_{\phi}$ in terms of
Picard categories (see Section~\ref{subsection-determ}). For a two-term complex
$C^{\bullet}=\{C^{-1}\stackrel{d}\to C^0\}$, let $\Picard(C^{\bullet})$ be the following Picard category:
objects in $\Picard(C^{\bullet})$ are elements in the group $C^0$ and morphisms
from $c\in C^0$ to $c'\in C^0$ are elements $\widetilde{c}\in C^{-1}$ such that
$d\widetilde{c}=c'-c$. A morphism of complexes $\phi$ defines
a morphism of complexes
$$
\tau_{\ge(p-1)}\tau_{\le p}A^{\bullet}\times \tau_{\ge-p}\tau_{\le(1-p)}A^{\bullet}\to N
$$
that defines a distributive functor
$$
F:\Picard(\tau_{\ge(p-1)}\tau_{\le p}A^{\bullet})\times
\Picard(\tau_{\ge-p}\tau_{\le(1-p)}A^{\bullet})\to BN,
$$
where $BN$ is a Picard category with one object $0$ and with $\pi_1(BN)=N$.
The functor $F$ is defined as follows: for morphisms $\widetilde{a}:a\to a'$ and
$\widetilde{b}:b\to b'$, we have $F(\widetilde{a},\widetilde{b})=
\phi(\widetilde{a}\otimes b')+\phi(a\otimes\widetilde{b})$. Consider the restriction
of $F$ to the full Picard subcategories that consist of objects whose cohomology
classes belong to $H^p(A^{\bullet})'$ and $H^{1-p}(B^{\bullet})'$, respectively.
By Lemma~\ref{lemma-Deligne}, we get a biextension $P_{\phi}$ of
$(H^p(A^{\bullet})',H^{1-p}(B^{\bullet})')$ by $N$.
\end{rmk}

\begin{rmk}
In what follows there will be given shifted pairings
\mbox{$p\in\Hom_{D^-(\Abc)}
(A^{\bullet}\otimes^L_{\Z}B^{\bullet},N[m])$}, $m\in \Z$; this gives
a biextension of $(H^p(A^{\bullet}),H^{-m+1-p}(B^{\bullet}))$ by
$N$.
\end{rmk}

\begin{examp}\label{examp-dg-Weilpairing}
Consider a unital DG-algebra $A^{\bullet}$ over $\Z$, an abelian
group $N$, and an integer $p$. Suppose that $A^i=0$ for $i>d$. Given
a homomorphism $H^d(A^{\bullet})\to N$, by $H^p(A^{\bullet})'\subset
H^p(A^{\bullet})$ denote the annulator of the group
$H^{d-p}(A^{\bullet})$ with respect to the induced pairing
$H^p(A^{\bullet})\times H^{d-p}(A^{\bullet})\to N$. Analogously,
define the subgroup $H^{d+1-p}(A^{\bullet})'\subset
H^{d+1-p}(A^{\bullet})$. By the above construction, this defines a
biextension $P$ of $(H^p(A^{\bullet})',H^{d+1-p}(A^{\bullet})')$ by
$N$. The associated Weil pairing $\phi_l:H^p(A^{\bullet})'_l\times
H^{d+1-p}(A^{\bullet})'_l\to N_l$ has the following interpretation
as a Massey triple product. For the triple $a\in
H^p(A^{\bullet})'_l$, $l\in H^0(A^{\bullet})$, and $b\in
H^{d+1-p}(A^{\bullet})'_l$, there is a well defined Massey triple
product
$$
m_3(a,l,b)\in H^d(A^{\bullet})/(a\cdot
H^{d-p}(A^{\bullet})+H^{p-1}(A^{\bullet})\cdot b).
$$
By condition, the image $\overline{m}_3(a,l,b)\in N_l$ of
$m_3(a,l,b)$ with respect to the map $H^d(A^{\bullet})\to N$ is well
defined. By Lemma \ref{lemma-Weil-pair-commut}, we have
\mbox{$\phi_l(a,b)=\overline{m}_3(a,l,b)$}.
\end{examp}

\begin{examp}\label{examp-sheavesbiext}
Let $Y$ be a Noetherian scheme of finite dimension $d$; then any
sheaf of abelian groups $\F$ on $Y$ has no non-trivial cohomology
groups with numbers greater than $d$ and there is a natural morphism
$R\Gamma(Y,\F)\to H^d(Y,\F)[-d]$ in $D^b(\Abc)$. Let $\F$, $\G$ be
sheaves of abelian groups on $Y$, $N$ be an abelian group, and $p\ge
0$ be an integer. Given a homomorphism $H^d(Y,\F\otimes\G)\to N$,
let $H^p(Y,\F)'$ be the annulator of $H^{d-p}(Y,\G)$ with respect to
the natural pairing $H^p(Y,\F)\times H^{d-p}(Y,\G)\to N$.
Analogously, we defined the subgroup $H^{d+1-p}(Y,\G)'\subset
H^{d+1-p}(Y,\G)$. Then the multiplication morphism \mbox{$m\in
\Hom_{D^b(\Abc)}(R\Gamma(Y,\F)\otimes_{\Z}^L
R\Gamma(Y,\G),R\Gamma(Y,\F\otimes\G))$} and the morphism
$R\Gamma(Y,\F\otimes\G)\to N[-d]$ lead to a biextension of
$(H^p(Y,\F)',H^{d+1-p}(Y,\G)')$ by $N$.
\end{examp}

\subsection{Poincar\'e biextension}\label{subsection-Poincare}

Consider an example to the construction described above. Let $\H$ be
the abelian category of integral mixed Hodge structures. For an
object $H$ in $\H$, by $H_{\Z}$ denote the underlying finitely
generated abelian group, by $W_{*}H_{\Q}$ denote the increasing
weight filtration, and by $F^{*}H_{\C}$ denote the decreasing Hodge
filtration, where $H_{\C}=H_{\Z}\otimes_{\Z}\C$. Recall that if
$H_{\Z}$ is torsion-free, then the complex $R\Hom_{\H}(\Z(0),H)$ is
canonically quasiisomorphic to the complex
$$
B(H)^{\bullet}:0\to W_0 H_{\Z}\oplus (F^0\cap W_0) H_{\C}\to
W_0H_{\C}\to 0,
$$
where the differential is given by $d(\gamma,\eta)=\gamma-\eta$ for
$\gamma\in W_0H_{\Z}$, $\eta\in (F^0\cap W_0)H_{\C}$, and
$W_0H_{\Z}=H_{\Z}\cap W_0 H_{\Q}$ (see \cite{Bei86}). Moreover, for
any two integral mixed Hodge structures $H$ and $H'$ with $H_{\Z}$
and $H'_{\Z}$ torsion-free, the natural pairing
$R\Hom_{\H}(\Z(0),H)\otimes^L_{\Z}\RHom_{\H}(\Z(0),H')\to
R\Hom_{\H}(\Z(0),H\otimes H')$ corresponds to the morphism
$$
m:B(H)^{\bullet}\otimes_{\Z} B(H')^{\bullet}\to B(H\otimes
H')^{\bullet}
$$
given by the formula $m(\varphi\otimes\eta')=\varphi\otimes\eta'\in
(H\otimes H')_{\C}$,
$m(\gamma\otimes\varphi')=\gamma\otimes\varphi'\in (H\otimes
H')_{\C}$, $m(\gamma\otimes\gamma')=\gamma\otimes\gamma'\in
W_0(H\otimes H')_{\Z}$, $m(\eta\otimes\eta')=\eta\otimes\eta'\in
F^0(H\otimes H')_{\C}$, and $m=0$ otherwise, where $\gamma\in W_0
H_{\Z}$, $\gamma'\in W_0 H_{\Z}'$, $\eta\in F^0H_{\C}$, $\eta'\in
F^0H'_{\C}$, $\varphi\in H_{\C}$, $\varphi'\in H'_{\C}$.

Let $H$ be a pure Hodge structure of weight $-1$, and let
$H^{\vee}=\Hom(H,\Z(1))$ be the corresponding internal $\Hom$ in
$\H$. By
$\langle\cdot,\cdot\rangle:H_{\C}\otimes_{\C}H^{\vee}_{\C}\to \C^*$
denote the composition of the natural map
$H_{\C}\otimes_{\C}H^{\vee}_{\C}\to \C$ and the exponential map
$\C\to \C/\Z(1)=\C^*$. We put
$J(H)=\Ext^1_{\H}(\Z(0),H)=H_{\C}/(H_{\Z}+F^0H_{\C})$. This group
has a natural structure of a compact complex torus; we call it the
{\it Jacobian of $H$}. Note that the Jacobian $J(H^{\vee})$ is
naturally isomorphic to the dual compact complex torus
$J(H)^{\vee}$.

By Section \ref{subsection-complexes}, the natural pairing
$$
R\Hom(\Z(0),H)\otimes^L_{\Z}R\Hom(\Z(0),H^{\vee})\to
R\Hom(\Z(0),H\otimes H^{\vee})\to
$$
$$
\to R\Hom(\Z(0),\Z(1))\to \C^*[-1]
$$
defines a biextension $P$ of $(J(H),J(H^{\vee}))$ by $\C^*$. The
given above explicit description of the pairing $m$ shows that the
biextension $P$ is induced by the bilinear map $\psi:S\to \C^*$ (see
Section \ref{subsection-quotientbiext}), where
$$
S=((H_{\Z}+F^0H_{\C})\times H^{\vee}_{\C})\cup (H_{\C}\times
(H_{\Z}^{\vee}+ F^0H^{\vee}_{\C}))
$$
is a bisubgroup in $H_{\C}\times H^{\vee}_{\C}$,
$\psi(\gamma+\eta,\varphi^{\vee})=\langle\gamma,\varphi^{\vee}\rangle$,
and
$\psi(\varphi,\gamma^{\vee}+\eta^{\vee})=\langle\varphi,\eta^{\vee}\rangle$
for all $\gamma\in H_{\Z}$, $\gamma^{\vee}\in H_{\Z}^{\vee}$,
$\eta\in F^0H_{\C}$, $\eta^{\vee}\in F^0H^{\vee}_{\C}$, $\varphi\in
H_{\C}$, $\varphi^{\vee}\in H^{\vee}_{\C}$.

\begin{rmk}\label{rmk-Hain}
The given above explicit construction of the biextension $P$ shows
that it coincides with the biextension constructed in \cite{Hai},
Section 3.2. In particular, by Lemma 3.2.5 from op.cit., $P$ is
canonically isomorphic to the Poincar\'e line bundle over
$J(H)\times J(H)^{\vee}$.
\end{rmk}

By Remark \ref{rmk-Hain}, it makes sense to call $P$ the {\it
Poincar\'e biextension}.

\begin{rmk}\label{rmk-Hain-canonical}
It follows from \cite{Hai}, Section 3.2 that the fiber of the
biextension $P$ over a pair $(e,f)\in J(H)\times J(H^{\vee})=
\Ext^1_{\H}(\Z(0),H)\times \Ext^1_{\H}(H,\Z(1))$ is canonically
bijective with the set of isomorphism classes of integral mixed
Hodge structures $V$ whose weight graded quotients are identified
with $\Z(0)$, $H$, $\Z(1)$ and such that $[V/W_{-2}V]=e\in
\Ext^1_{\H}(\Z(0),H)$, $[W_{-1}V]=f\in \Ext^1_{\H}(H,\Z(1))$.
\end{rmk}

\begin{rmk}\label{rmk-height}
There is a canonical trivialization of the biextension $\log|P|$ of
$(J(H),J(H^{\vee}))$ by $\R$; this trivialization is given by the
bilinear map $H_{\C}\times H^{\vee}_{\C}\to \R$,
$(\varphi,\varphi^{\vee})\mapsto \log|\langle
r,\varphi^{\vee}\rangle|$, where $\varphi=r+\eta$, $r\in H_{\R}$,
$\eta\in F^0H$.
\end{rmk}

\section{Biextensions over Chow groups}\label{section-biextChow}

\subsection{Explicit construction}\label{subsection-acycles}

The construction described below was first given in \cite{Blo89}. We
use notions and notations from Section~\ref{subsection-K1}. Let $X$
be a smooth projective variety of dimension $d$ over $k$. For an
integer $p\ge 0$, by $Z^p(X)'$ denote the subgroup in $Z^p(X)$ that
consists of cycles $Z$ such that for any $K_1$-chain
$\{f_{\eta}\}\in G^{d-p}_{|Z|}(X,1)$ with $\div(\{f_{\eta}\})=0$, we
have $\prod_{\eta}f_{\eta} (\overline{\eta}\cap Z)=1$. Let
$CH^p(X)'$ be the image of $Z^p(X)'$ in $CH^p(X)$; by Lemma
\ref{lemma-Bloch}, we have $CH^p(X)_{\hhom}\subset CH^p(X)'$.

\begin{rmk}\label{rmk-numtriv}
If the group $k^*$ has an element of infinite order, then it is easy to see that
any cycle $Z\in Z^p(X)'$ is numerically trivial.
\end{rmk}

Consider integers $p,q\ge 0$ such that $p+q=d+1$. Let $T\subset
Z^p(X)'\times Z^q(X)'$ be the bisubgroup that consists of pairs
$(Z,W)$ such that $|Z|\cap |W|=\emptyset$. By the classical moving
lemma, we see that the map $T\to CH^p(X)'\times CH^q(X)'$ is
surjective. Let $S\subset T$ be the corresponding bisubgroup as
defined in Section~\ref{subsection-quotientbiext}. Define a bilinear
map $\psi:S\to k^*$ as follows. Suppose that $Z$ is rationally
trivial; then by Lemma \ref{lemma-movingK1}, there exists an element
$\{f_{\eta}\}\in G_{|W|}^{p-1}(X,1)$ such that
$\div(\{f_{\eta}\})=Z$. We put
$\psi(Z,W)=\prod_{\eta}f_{\eta}(\overline{\eta}\cap W)$. By
condition, this element from $k^*$ does not depend on the choice of
$\{f_{\eta}\}$. Similarly, we put $\psi(Z,W)=
\prod_{\xi}g_{\xi}(Z\cap\overline{\xi})$ if $W=\div(\{g_{\xi}\})$,
$\{g_{\xi}\}\in G^{q-1}_{|Z|}(X,1)$. By Lemma~\ref{lemma-recipr},
$\psi(\div(\{f_{\eta}\}),\div(\{g_{\xi}\}))$ is well defined. The
application of the construction from
Section~\ref{subsection-quotientbiext} yields a biextension $P_E$ of
$(CH^p(X)',CH^q(X)')$ by $k^*$.

The following interpretation of the biextension $P_E$ from
Section~\ref{subsection-acycles} in terms of pairings between higher
Chow complexes was suggested to the author by S.\,Bloch.

We use notations and notions from Section~\ref{subsection-higherChow}.
As above, let $X$ be a smooth projective variety of dimension $d$ over $k$ and
let the integers $p,q\ge 0$ be such that $p+q=d+1$. There is a push-forward morphism of complexes
$\pi_*:Z^{d+1}(X,\bullet)\to
Z^{1}(\Spec(k),\bullet)$, where $\pi:X\to\Spec(k)$ is the structure map. Further, there is a morphism of complexes
$Z^{1}(\Spec(k),\bullet)\to k^*[-1]$ (which is actually a quasiisomorphism).
Taking the composition of these morphisms with the multiplication morphism
$$
m\in \Hom_{D^-(\Abc)}(Z^p(X,\bullet)\otimes_{\Z}^L Z^q(X,\bullet),Z^{d+1}(X,\bullet)),
$$
we get an element
$$
\phi\in\Hom_{D^-(\Abc)}(Z^p(X,\bullet)\otimes_{\Z}^L Z^q(X,\bullet),
k^*[-1]).
$$
Note that by Lemma \ref{lemma-moving},
in notations from Section \ref{subsection-complexes},
we have $H_0(Z^p(X,\bullet))'=CH^p(X)'$ and $H_0(Z^q(X,\bullet))'=CH^q(X)'$.
Hence the construction from Section \ref{subsection-complexes}
gives a biextension $P_{HC}$ of $(CH^p(X)',CH^q(X)')$ by $k^*$.

\begin{prop}\label{prop-Chowbiext}
The biextension $P_{HC}$ is canonically isomorphic to the
biextension $P_E$.
\end{prop}
\begin{proof}
Recall that the multiplication morphism $m$ is given by the composition
$$
Z^p(X,\bullet)\otimes Z^q(X,\bullet)\stackrel{ext}\to Z^{d+1}(X\times X,\bullet)
\hookleftarrow Z^{d+1}_{D}(X\times X,\bullet)\stackrel{D^*}\to Z^{d+1}(X,\bullet),
$$
where $D\subset X\times X$ is the diagonal.
Let $A_0\subset Z^p(X)\otimes Z^q(X)$ be the subgroup generated by elements $Z\otimes W$
such that $|Z|\cap |W|$, and let
$$
A_1\subset Z^p(X,1)\otimes Z^q(X,0)\oplus
Z^p(X,0)\otimes Z^q(X,1)=(Z^p(X,\bullet)\otimes Z^q(X,\bullet))_1
$$
be the subgroup generated by elements
$(\alpha\otimes W,Z\otimes\beta)$ such that $\alpha\in Z^p_{|W|}(X,1)$,
$\beta\in Z^q_{|Z|}(X,1)$; we have $ext(A_i)\subset Z^{d+1}_{D}(X\times X,i)$
for $i=0,1$.
Since all terms of the complex $Z^p(X,\bullet)\otimes Z^q(Z,\bullet)$
are free $\Z$-modules and
the groups $A_i$ are direct summands in the groups
$(Z^p(X,\bullet)\otimes Z^q(X,\bullet))_i$ for $i=0,1$, there exists a true morphism
of complexes $ext':Z^p(X,\bullet)\otimes
Z^q(X,\bullet)\to Z_D^{p+q}(X\times X,\bullet)$
such that $ext'$ is equivalent to $ext$ in the derived category $D^-(\Abc)$
and $ext'$ coincides with $ext$ on $A_i$, $i=0,1$.

On the other hand, for a cycle $Z\in Z^p(X)$ and a $K_1$-chain
$\{g_{\xi}\}\in G^{q-1}_{|Z|}(X,1)$, we have $\beta=\Gamma(\{g_{\xi}\})\in
Z^q_{|Z|}(X,1)$ and $(\pi_*\circ D^*\circ ext)(Z\otimes\beta))=
\prod_{\xi}g_{\xi}(Z\cap\overline{\xi})$ (see Section~\ref{subsection-K1}).

Therefore the needed statement follows from Remark~\ref{rmk-biextcoinc}
applied to the bisubgroup $T\subset Z^p(X)\times Z^q(X)$
together with the bilinear map $\psi$ from Section~\ref{subsection-acycles}
and the bisubgroup in $Z^p(X)\times Z^q(X)$ together the bilinear map
induced by the morphism of complexes $\phi=\pi_*\circ D^*\circ ext'$ as described in
Section~\ref{subsection-complexes}.
\end{proof}

\subsection{Intermediate Jacobians construction}\label{subsection-AJbiext}

The main result of this section was partially proved in \cite{MS95}
by using the functorial properties of higher Chow groups and the
regulator map to Deligne cohomology. We give a more elementary proof
that uses only the basic Hodge theory.

We use notions and notations from Sections \ref{subsection-Poincare}
and \ref{subsection-Abel-Jacobi}. Let $X$ be a complex smooth
projective variety of dimension $d$. Suppose that $p$ and $q$ are
natural numbers such that $p+q=d+1$. Multiplication by $(2\pi i)^p$
induces the isomorphism $J^{2p-1}(X)\to J(H^{2p-1}(X)(p))$. Since
$H^{2p-1}(X)(p)^{\vee}=H^{2q-1}(X)(q)$, by Section
\ref{subsection-Abel-Jacobi}, there is a Poincar\'e biextension
$P_{IJ}$ of $(J^{2p-1}(X),J^{2q-1}(X))$ by $\C^*$.

\begin{prop}\label{prop-AbelJacobi}
The pull-back $AJ^*(P_{IJ})$ is canonically isomorphic to the
restriction of the biextension $P_E$ constructed in section
\ref{subsection-acycles} to subgroups $CH^p(X)_{\hom}\subset
CH^p(X)'$.
\end{prop}
\begin{proof}[Proof 1]
First, let us give a short proof that uses functorial properties of
the regulator map from the higher Chow complex to the Deligne
complex.

Let $\Z_D(p)^{\bullet}={\rm cone}(\Z(p)\oplus
F^p\Omega_X^{\bullet}\to \Omega_X^{\bullet})[-1]$ be the Deligne
complex, where all sheaves are considered in the analytic topology.
There are a multiplication morphism
$$
R\Gamma_{an}(X(\Cb),\Z_D(p)^{\bullet})
\otimes^L_{\Z}R\Gamma_{an}(X(\Cb),\Z_D(q)^{\bullet})\to
R\Gamma_{an}(X(\Cb),\Z_D(p+q)^{\bullet})
$$
and a push-forward morphism
$$
R\Gamma_{an}(X(\Cb),\Z_D(d+1)^{\bullet})\to \Cb^*[-2d-1]
$$
in $D^b(\Abc)$. The application of the construction from Section
\ref{subsection-complexes} to the arising pairing of complexes
$$
R\Gamma_{an}(X(\C),\Z_D(p)^{\bullet}) \otimes^L_{\Z}
R\Gamma_{an}(X(\C),\Z_D(q)^{\bullet})\to\Cb^*[-2d-1]
$$
for $p+q=d+1$ yields a biextension of $(J^{2p-1}(X),J^{2q-1}(X))$ by $\Cb^*$,
since $J^{2p-1}(X)\subset H^{2p}_{an}(X(\Cb),\Z_D(p)^{\bullet})'$
and $J^{2q-1}(X)\subset H^{2q}_{an}(X(\Cb),\Z_D(q)^{\bullet})'$.

Furthermore, there is a canonical morphism
$$
\epsilon:R\Hom_{D^b(\H)}(\Z(0),RH^{\bullet}(X)(p))\to
R\Gamma_{an}(X(\Cb),\Z_D(p)^{\bullet}),
$$
where $RH^{\bullet}(X)$ is the ``derived'' Hodge structure of $X$.
The morphism $\epsilon$ induces isomorphisms on cohomology groups in degrees
less or equal to $2p$ (this
follows from the explicit formula for the complex $B(H)^{\bullet}$
given in Section~\ref{subsection-Poincare}). Moreover, the
morphism $\epsilon$ commutes with the multiplication and the push-forward.
Note that the multiplication morphism on the left hand side is defined as the composition
$$
R\Hom_{D^b(\H)}(\Z(0),RH^{\bullet}(X)(p))
\otimes^L_{\Z}R\Hom_{D^b(\H)}(\Z(0),RH^{\bullet}(X)(q))\to
$$
$$
R\Hom_{D^b(\H)}(\Z(0),RH^{\bullet}(X)\otimes RH^{\bullet}(X)(p+q))\to
R\Hom_{D^b(\H)}(\Z(0),RH^{\bullet}(X)(p+q))
$$
and the
push-forward is defined as the composition
$$
R\Hom_{D^b(\H)}(\Z(0),RH^{\bullet}(X)(d+1))\to R\Hom_{D^b(\H)}(\Z(0),H^{2d}(X)(d+1)[-2d])=
$$
$$
=R\Hom_{D^b(\H)}(\Z(0),\Z(1))[-2d]=\Cb^*[-2d-1].
$$
Therefore, the biextension of $(J^{2p-1}(X),J^{2q-1}(X))$ by $\Cb^*$ defined via the pairing
of complexes on the left hand side is canonically isomorphic to the previous one.

The multiplicativity of the spectral sequence
$$
\Ext^i_{\H}(\Z(0),H^j(X)(p))\Rightarrow
\Hom^{i+j}_{D^b(\H)}(\Z(0),RH^{\bullet}(X)(p))
$$
shows that the last biextension is canonically isomorphic to the Poicar\'e
biextension defined in Section~\ref{subsection-Poincare}.

Now consider the cohomological type complex
\mbox{$Z^p(X)^{\bullet}=Z^p(X,2p-\bullet)$} built out of the higher
Chow complex (see Section \ref{subsection-higherChow}). In
\cite{KLM} it is constructed an explicit morphism
$$
\rho:Z^p(X)^{\bullet}\to R\Gamma_{an}(X(\Cb),\Z_D(p)^{\bullet})
$$
of objects in the derived category $D^-(\Abc)$, which is actually the
regulator map. It follows from op.cit. that the regulator map $\rho$
commutes with the corresponding multiplication morphisms for higher
Chow complexes and Deligne complexes. By
Proposition~\ref{prop-Chowbiext}, this implies immediately the
needed result.
\end{proof}

\begin{quest}
Does there exist a morphism of DG-algebras $\rho:A_M^{\bullet}\to
A_D^{\bullet}$ such that a DG-algebra $A_M^{\bullet}$ is
quasiisomorphic to the DG-algebra $\bigoplus\limits_{p\ge
0}R\Gamma_{Zar}(X,\Z(p)^{\bullet})$, where $\Z(p)^{\bullet}$ is the
Suslin complex, and a DG-algebra $A_D^{\bullet}$ is quasiisomorphic
to the DG-algebra $\bigoplus\limits_{p\ge
0}R\Gamma_{an}(X,\Z_D(p)^{\bullet})$? The identification of higher
Chow groups with motivic cohomology $CH^p(X,n)=H^{2p-n}_M(X,\Z(p))$
shows that a positive answer to the question would be implied by the
existence of a DG-realization functor from the DG-category of
Voevodsky motives (see \cite{BV}) to the DG-category associated with
integral mixed Hodge structures. It was pointed out to the author by
V.\,Vologodsky that the existence of this DG-realization functor
follows immediately from op.cit. Does there exist an explicit
construction of the DG-algebras $A_M^{\bullet}$, $A_D^{\bullet}$,
and the morphism $\rho$?
\end{quest}

\begin{proof}[Proof 2]
Let $\widetilde{Z}^p(X)_{\hom}$ be the group that consists of all
triples $\widetilde{Z}=(Z,\Gamma_Z,\eta_Z)$, where $Z\in
Z^p(X)_{\hom}$, $\Gamma_Z$ is a differentiable singular chain on $X$
such that $\partial\Gamma_Z=Z$, and \mbox{$\eta_Z\in
F_{\log}^pA^{2p-1}_{X\backslash|Z|}$} is such that
$\partial_Z([(2\pi i)^p\eta_Z])=[Z]$. It follows from Section
\ref{subsection-Abel-Jacobi} that the natural map
$\widetilde{Z}^p(X)_{\hom}\to Z^p(X)_{\hom}$ is surjective. There is
a homomorphism $\widetilde{Z}^p(X)\to H^{2p-1}(X,\C)$ given by the
formula $\widetilde{Z} \mapsto PD[\Gamma_Z]-[\eta_Z]$; its
composition with the natural map $H^{2p-1}(X,\C)\to J^{2p-1}(X)$ is
equal to the Abel--Jacobi map $\widetilde{Z}^p(X)_{\hom}\to
J^{2p-1}(X)$ given by the formula $\widetilde{Z}\mapsto AJ(Z)$ (see
section \ref{subsection-Abel-Jacobi}).

Let $T$ be a bisubgroup in $\widetilde{Z}^p(X)_{\hom}\times
\widetilde{Z}^q(X)_{\hom}$ that consists of all pairs of triples
$(\widetilde{Z},\widetilde{W})=((Z,\Gamma_Z,\eta_Z),(W,\Gamma_W,\eta_W))$
such that $|Z|\cap |W|=\emptyset$, $\Gamma_Z$ does not meet $|W|$,
and $|Z|$ does not meet $\Gamma_W$. Note that for any pair
$(\widetilde{Z},\widetilde{W})\in T$, there are well defined classes
$PD[\Gamma_Z],[\eta_Z]\in H^{2p-1}(X\backslash|Z|,|W|;\C)$ and
$PD[\Gamma_W],[\eta_W]\in H^{2q-1}(X\backslash|W|,|Z|;\C)$. As
before, denote by
$$
(\cdot,\cdot):H^{*}(X\backslash|Z|,|W|;\C)\times
H^{2d-*}(X\backslash|W|,|X|;\C)\to \C
$$
the natural pairing.

Let $S\subset T$ be a bisubgroup that consists of all pairs of
triples $(\widetilde{Z},\widetilde{W})$ such that $Z$ or $W$ is
rationally trivial. Let $\psi_1:S\to \C^*$ be the pull-back of the
bilinear map constructed in Section \ref{subsection-acycles} via the
natural map $\widetilde{Z}^p(X)_{\hom}\times
\widetilde{Z}^q(X)_{\hom}\to Z^p(X)_{\hom}\times Z^q(X)_{\hom}$. Let
$\psi_2:S\to \C^*$ be the pull-back of the bilinear constructed in
Section \ref{subsection-Poincare} via the defined above map
$\widetilde{Z}^p(X)_{\hom}\times \widetilde{Z}^q(X)_{\hom}\to
H^{2p-1}(X,\C)\times H^{2q-1}(X,\C)$. The construction from Section
\ref{subsection-quotientbiext} applied to the maps $\psi_1$ and
$\psi_2$ gives the biextensions $P_E$ and $P_{IJ}$, respectively.
Define a bilinear map $\phi:T\to \C^*$ by the formula
$\phi(\widetilde{Z},\widetilde{W})=\exp(2\pi
i\int_{\Gamma_Z}\eta_W)$. By Remark \ref{rmk-isombiext}, it is
enough to show that $\psi_1=\phi|_S\cdot\psi_2$.

Consider a pair $(\widetilde{Z},\widetilde{W})\in S$. First, suppose
that $Z$ is rationally trivial. By linearity and Lemma
\ref{lemma-movingK1}, we may assume that $Z=\div(f)$, where $f\in
\C^*(Y)$ and $Y\subset X$ is an irreducible subvariety of
codimension $p-1$ such that $Y$ meets $|W|$ properly. Let
$\widetilde{Y}$ be the closure of the graph of the rational function
$f:Y\dasharrow\P^1$ and let $p:\widetilde{Y}\to X$ be the natural
map. Let $\Gamma$ be a smooth generic path on $\P^1$ such that
$\partial\Gamma=\{0\}-\{\infty\}$ and $\Gamma$ does not intersect
with the finite set $f(p^{-1}(W))$; there is a cohomological class
$PD[\Gamma]\in H^1(\P^1\backslash\{0,\infty\},f(p^{-1}(|W|));\Z)$.
We put $\alpha_Z=(2\pi i)^{-p}[\widetilde{Y}]^*(2\pi iPD[\gamma])\in
H^{2p-1}(X\backslash|Z|,|W|;\Z)$. By Lemma \ref{lemma-correspAJ}
with $X_1=\P^1$, $X_2=X$, $C=\widetilde{Y}$, $Z_1=\{0,\infty\}$, and
$W_2=|W|$, we get $PD[\Gamma_Z]-\alpha_Z\in H^{2p-1}(X,|W|;\Z)$ and
$\alpha_Z-[\eta_Z]\in F^pH^{2p-1}(X,\Cb)$ (note the element
$\alpha_Z-[\eta_Z]\in H^{2p-1}(X,|W|;\Cb)$ does not necessary belong
to the subgroup $F^pH^{2p-1}(X,|W|;\Cb)$). Since
$PD[\Gamma_Z]-[\eta_Z]=(PD[\Gamma_Z]-\alpha_Z)+(\alpha_Z-[\eta_Z])$,
it follows from the explicit construction given in Section
\ref{subsection-Poincare} that
$$
\psi_2(\widetilde{Z},\widetilde{W})=\exp(2\pi
i(PD[\Gamma_Z]-\alpha_Z, PD[\Gamma_W]-[\eta_W]))=
$$
$$
\mbox{$\exp(-2\pi i\int_{\Gamma_Z}\eta_W+2\pi
i(\alpha_Z,[\eta_W])).$}
$$
In addition, combining Lemma \ref{lemma-complan}$(i)$ and Lemma
\ref{lemma-correspAJ} with $X_1=X$, $X_2=\P^1$, $C=\widetilde{Y}$,
$Z_1=|W|$, and $W_2=\{0,\infty\}$, we see that $\exp(2\pi
i(\alpha_Z,[\eta_W]))=f(Y\cap
W)=\psi_1(\widetilde{Z},\widetilde{W})$.

Now suppose that $W$ is rationally trivial. As before, by linearity
and Lemma \ref{lemma-movingK1}, we may assume that $W=\div(g)$,
where $W\in \C^*(V)$ and $V\subset X$ is an irreducible subvariety
of codimension $p-1$ such that $|Z|$ meets $V$ properly. Let
$\widetilde{V}$ be the closure of the graph of the rational function
$g:V\dasharrow\P^1$ and let $p:\widetilde{V}\to X$ be the natural
map. Note that the differential form $\frac{dz}{z}$ defines a
cohomological class $[\frac{dz}{z}]\in
F^1H^1(\P^1\backslash\{0,\infty\},g(p^{-1}(|Z|));\Cb)$. We put
$\alpha_W=(2\pi i)^{-q}[\widetilde{V}]^*[\frac{dz}{z}]\in
F^qH^{2q-1}(X\backslash|W|,|Z|;\C)$. By Lemma \ref{lemma-correspAJ}
with $X_1=\P^1$, $X_2=X$, $C=\widetilde{V}$, $Z_1=\{0,\infty\}$, and
$W_2=|Z|$, we get $\alpha_W-[\eta_W]\in F^qH^{2q-1}(X,|Z|;\Z)$ and
$PD[\Gamma_W]-\alpha_W\in H^{2q-1}(X,\Z)$ (note that the element
$PD[\Gamma_W]-\alpha_W\in H^{2q-1}(X,|Z|;\Cb)$ does not necessary
belong to the subgroup $H^{2q-1}(X,|Z|;\Z)$). Since
$PD[\Gamma_W]-[\eta_W]=(PD[\Gamma_W]-\alpha_W)+(\alpha_W-[\eta_W])$,
it follows from the explicit construction given in Section
\ref{subsection-Poincare} that
$$
\psi_2(\widetilde{Z},\widetilde{W})= \exp(2\pi
i(PD[\Gamma_Z]-[\eta_Z],\alpha_W-[\eta_W]))=
$$
$$
\mbox{$=\exp(2\pi i(PD[\Gamma_Z],\alpha_W)-2\pi
i\int_{\Gamma_Z}\eta_W)$}.
$$
In addition, combining Lemma \ref{lemma-complan}$(ii)$ and Lemma
\ref{lemma-correspAJ} with $X_1=X$, $X_2=\P^1$, $C=\widetilde{Y}$,
$Z_1=|Z|$, and $W_2=\{0,\infty\}$, we see that $\exp(2\pi
i(PD[\Gamma_Z],\alpha_W))=g(Z\cap
Y)=\psi_1(\widetilde{Z},\widetilde{W})$. This concludes the proof.
\end{proof}

During the proof of Proposition~\ref{prop-AbelJacobi} we have used
the following simple fact.

\begin{lemma}\label{lemma-complan}
\hspace{0cm}
\begin{itemize}
\item[(i)]
Let $\eta$ be a meromorphic $1$-form on $\P^1$ with poles of order
at most one (i.e., $\eta$ is a differential of the third kind),
$\Gamma$ be a smooth generic path on $\P^1$ such that
$\partial\Gamma=\{0\}-\{\infty\}$ and $\Gamma$ does not contain any
pole of $\eta$; if ${\rm res}(2\pi i\eta)=\sum_i{n_i}\{z_i\}$ for
some integers $n_i$, then we have $\mbox{$\exp(2\pi
i\int_{\Gamma}\eta)=\prod_i z_i^{n_i}$}$.
\item[(ii)]
Let $z$ be a coordinate on $\P^1$, $\Gamma$ be a differentiable
singular $1$-chain on $\P^1$ that does not intersect with the set
$\{0,\infty\}$; if $\partial\Gamma=\sum_i n_i\{z_i\}$, then we have
$\mbox{$\exp(\int_{\Gamma}\frac{dz}{z})=\prod_i z_i^{n_i}$}$.
\end{itemize}
\end{lemma}
\begin{proof}
$(i)$ Let $\log z$ be a branch of logarithm on $\P^1\backslash
\Gamma$, $T_{\epsilon}$ be a tubular neighborhood of $\Gamma$ with
radius $\epsilon$, and let $B_{i,\epsilon}$ be disks around
$\{z_i\}$ with radius $\epsilon$; we put
$X_{\epsilon}=\P^1\backslash (\cup_iB_{i,\epsilon}\cup
T_{\epsilon})$. Then we have $0=\lim\limits_{\epsilon\to 0}
\int_{X_{\epsilon}} d((\log z) \eta)=\sum_i n_i\log z_i-2\pi
i\int_{\Gamma}\eta$; this concludes the proof.

$(ii)$ One may assume that $\Gamma$ does not contain any loop around
$\{0\}$ or $\{\infty\}$; hence there exists a smooth path $\gamma$
on $\P^1$ such that $\partial\gamma=\{0\}-\{\infty\}$ and $\gamma$
does not intersect $\Gamma$. Let $\log z$ be a branch of logarithm
on $\P^1\backslash\gamma$. Then we have
$\int_{\Gamma}\frac{dz}{z}=\int_{\Gamma}d\log z=\sum_i n_i\log z_i$;
this concludes the proof.
\end{proof}

\begin{rmk}\label{rmk-explicPoinccycles}
For a cycle $Z\in Z^p(X)_{\hom}$, consider an exact sequence of
integral mixed Hodge structures
$$
0\to H^{2p-1}(X)(p)\to H^{2p-1}(X\backslash |Z|)(p)\to
H_{2d-2p}(|Z|)\to H^{2p}(X)(p).
$$
Its restriction to $[Z]_{\Z}=\Z(0)\subset H_{2d-2p}(|Z|)$ defines a
short exact sequence
$$
0\to H^{2p-1}(X)(p)\to E_Z\to \Z(0)\to 0.
$$
Analogously, for a cycle $W\in Z^q(X)_{\hom}$, we get an exact
sequence
$$
0\to H^{2q-1}(X)(q)\to E_W\to \Z(0)\to 0
$$
and a dual exact sequence
$$
0\to \Z(1)\to E_W^{\vee}\to H^{2p-1}(X)(p)\to 0.
$$
By Remark~\ref{rmk-Hain-canonical}, we see that the fiber
$P_{IJ}|_{(AJ(Z),AJ(W))}$ is canonically bijective with the set of
isomorphism classes of all integral mixed Hodge structures $V$ whose
weight graded quotients are identified with $\Z(0)$,
$H^{2p-1}(X)(p)$, $\Z(1)$ and such that $[V/W_{-2}V]=[E_Z]\in
\Ext^1_{\H}(\Z(0),H^{2p-1}(X)(p))$, $[W_{-1}V]=[E_{W}^{\vee}]\in
\Ext^1_{\H}(H^{2p-1}(X)(p),\Z(1))$.
\end{rmk}

\begin{rmk}
Under the assumptions of Corollary \ref{corol-canon-torsors} suppose
that $|Z|\cap |W|=\emptyset$. Then there are two canonical
trivializations of the fiber of the biextension $\log|P_E|$ over
$(Z,W)$: the first one follows from the construction of $P_E$ and
the second one follows from Proposition \ref{prop-AbelJacobi} and
Remark \ref{rmk-height}. Let $h(Z,W)\in\R$ be the quotient of these
two trivializations. Consider closed forms $\eta^r_Z\in
F^p_{\log}A^{2p-1}_{X\backslash|Z|}$ and $\eta^r_W\in
F^q_{\log}A^{2q-1}_{X\backslash|W|}$ such that $\partial_Z ((2\pi
i)^p\eta^r_Z)=[Z]$, $\partial_W( (2\pi i)^q\eta^r_W)=[W]$,
$\eta^r_Z$ has real periods on $X\backslash|Z|$, and $\eta_W^r$ has
real periods on $X\backslash |W|$ (it is easy to see that such forms
always exist). In notations from the proof of
Proposition~\ref{prop-AbelJacobi} the isomorphism $P_E\to P_{IJ}$ is
given by the multiplication by $\exp(2\pi i\int_{\Gamma_Z}\eta_W)$.
On the other hand, the trivialization from Remark \ref{rmk-height}
is given by the section ${\rm Re}(2\pi
i(PD[\Gamma_Z]-[\eta^r_Z],PD[\Gamma_W]-[\eta_W]))= {\rm Re}(2\pi
i\int_{\Gamma_Z}(\eta^r_W-\eta_W))$. Therefore, $h(Z,W)={\rm
Re}(2\pi i\int_{\Gamma_Z}\eta^r_W)$. This number is called the {\it
archimedean height pairing} of $Z$ and $W$.
\end{rmk}

Finally, using Lemma \ref{lemma-complan} we give an explicit
analytic proof of Lemma \ref{lemma-Bloch} for complex varieties.

\begin{lemma}\label{lemma-BlochAnal}
Let $X$ be a complex smooth projective variety of dimension $d$,
$W\in Z^q(X)_{\hom}$ be a homologically trivial cycle, $\{Y_i\}$ be
a finite collection of irreducible subvarieties of codimension $d-q$
in $X$, and let $f_i\in \C(Y_i)^*$ be a collection of rational
functions such that $\sum_i\div(f_i)=0$ and for any $i$, the support
$|W|$ meets $Y_i$ properly and $|W|\cap |\div(f_i)|=\emptyset$. Then
we have $f(Y\cap W)=\prod_i f_i(Y_i\cap W)=1$; here for each $i$, we
put $f_i(Y_i\cap W)=\prod\limits_{x\in Y_i\cap
|W|}f_i^{(Y_i,W;x)}(x)$, where $(Y_i,W;x)$ is the intersection index
of $Y_i$ and $W$ at a point $x\in Y_i\cap |W|$.
\end{lemma}
\begin{proof}
Since $W$ is homologically trivial, by Section
\ref{subsection-Abel-Jacobi}, there exists a closed differential
form $\eta\in F^q_{\log}A^{2q-1}_{X\backslash|W|}$ such that
$\partial_W([(2\pi i)^q\eta])=[W]$, i.e., for a small
$(2q-1)$-sphere $S_j$ around an irreducible component $W_j$ of $W$,
we have $\int_{S_j}\eta=m_j$, where $W=\sum m_jW_j$.

For each $i$, let $\widetilde{Y}_i$ be the closure of the graph of
the rational function $f_i:Y_i\dasharrow\P^1$ and let
$p_i:\widetilde{Y}_i\to X$ be the natural map. Let $\gamma$ be a
smooth generic path on $\P^1$ such that
$\partial\gamma=\{0\}-\{\infty\}$ and $\gamma$ does not intersect
with the finite set $\Sigma=\cup_i p_i(f_i^{-1}(|W|))$. There is a
cohomological class $PD[\gamma]\in
H^1(\P^1\backslash\{0,\infty\},\Sigma;\Z)$; we put $\alpha_i= (2\pi
i)^{-(d-q+1)}[\widetilde{Y}_i]^*(2\pi i PD[\gamma])\in
H^{2d-2q+1}(X\backslash|\div(f_i)|,|W|;\Z)$. Let $\Gamma_i$ be a
differentiable representative for $PD(\alpha_i)\in
H_{2q-1}(X\backslash|W|,|\div(f_i)|;\Z)$ (see Remark
\ref{rmk-explicit-chain}). In particular, $\Gamma_i$ is a
differentiable singular $(2q-1)$-chain on $X\backslash |W|$ with
boundary $\div(f_i)$.

By Lemma \ref{lemma-complan} and Lemma \ref{lemma-correspAJ} with
$X_1=\P^1$, $X_2=X$, $C=\widetilde{Y}_i$, $Z_1=\{0,\infty\}$, and
$W_2=|W|$, we get $f_i(Y_i\cap W)=\exp(2\pi
i(\alpha_i,\eta))=\exp(2\pi i\int_{\Gamma_i}\eta)$. Therefore,
$f(Y\cap W)=\exp(2\pi i\int_{\Gamma}\eta)$, where $\Gamma=\sum_i
\Gamma_i$. Since $\sum_i\div(f_i)=0$, the chain $\Gamma$ on
$X\backslash |W|$ has no boundary.

Note that the image of the class $\alpha_i$ under the natural
homomorphism
$$
H^{2d-2q+1}(X\backslash|\div(f_i)|,|W|;\C)\to
H^{2d-2q+1}(X\backslash|\div(f_i)|,\C)
$$
belongs to the subgroup
$F^{d-q+1}H^{2d-2q+1}(X\backslash|\div(f_i)|,\C)$, i.e, for any
closed form $\omega\in F^{q}A^{2q-1}_X$, we have
$\int_{\Gamma_i}\omega=0$. It follows that the chain $\Gamma$ is
homologous to zero on $X$. Therefore, $\Gamma$ is homologous on
$X\backslash |W|$ to an integral linear combination of small
$(2q-1)$-spheres around irreducible components of the subvariety
$|W|$; hence $\int_{\Gamma}\eta\in\Z$ and this concludes the proof.
\end{proof}

\subsection{$K$-cohomology construction}\label{subsection-Kcohombiext}

We use notions
and notations from Section~\ref{subsection-adelic}.
Let $X$ be a smooth projective variety of dimension $d$
over a field $k$, and let $p,q\ge 0$ be such that $p+q=d+1$.
There is a product morphism
$m:\K_p\otimes\K_q\to \K_{d+1}$ and a push-forward map
$\pi_*:H^d(X,\K_{d+1})\to k^*$, where $\pi:X\to\Spec(k)$ is the
structure map. By Example \ref{examp-sheavesbiext}, for any integer
$p\ge 0$, we get a biextension of $(H^p(X,\K_p)',H^{q}(X,\K_{q})')$
by $k^*$, where $H^p(X,\K_p)'\subset H^p(X,\K_p)$ is the annulator
of the group $H^{q-1}(X,\K_{q})$ with respect to the pairing
$H^p(X,\K_p)\times H^{q-1}(X,\K_q)\to k^*$ and the analogous is true
for the subgroup $H^q(X,\K_q)'\subset H^q(X,\K_q)$.

Consider a cycle $W\in Z^q(X)$ and a $K_1$-chain $\{f_{\eta}\}\in
G^{d-q}_{|W|}(X,1)$ such that $\div(\{f_{\eta}\})=0$. By
$\overline{W}$ and $\overline{\{f_{\eta}\}}$ denote the classes of
$W$ and $\{f_{\eta}\}$ in the corresponding $K$-cohomology groups.
It follows directly from Lemma \ref{lemma-adelicproduct} that
$\pi_*(m(\overline{\{f_{\eta}\}}\otimes \overline{W}))=
\prod_{\eta}f_{\eta}(\overline{\eta}\cap W)$. Hence by Lemma
\ref{lemma-Bloch}, the identification $CH^p(X)=H^p(X,\K_p)$ induces
the equality $CH^p(X)'= H^p(X,\K_p)'$; thus we get a biextension
$P_{KC}$ of $(CH^p(X)',CH^{d+1-1}(X)')$ by $k^*$.

\begin{prop}\label{prop-biext-Kcohom}
The biextension $P_{KC}$ is canonically isomorphic up to the sign
$(-1)^{pq}$ to the biextension $P_E$.
\end{prop}
\begin{proof}[Proof 1]
First, we give a short proof that uses a regulator map from higher
Chow groups to $K$-cohomology. The author is very grateful to the referee
who suggested this proof.

Consider the cohomological type higher Chow complex
$$
Z^p(X)^{\bullet}=Z^p(X,2p-\bullet)
$$
and the complex of Zariski flabby sheaves
$\underline{Z}^p(X)^{\bullet}$ defined by the formula
$\underline{Z}^p(X)^{\bullet}(U)=Z^p(U)^{\bullet}$ for an open
subset $U\subset X$. Let $\Hc^i$ be the cohomology sheaves of the
complex $\underline{Z}^p(X)^{\bullet}$. It follows from \cite{Blo86}
combined with \cite{NS} and \cite{Tot} that $\Hc^i=0$ for $i>p$ and
that the sheaf $\Hc^p$ has the following flabby resolution
$\underline{Gers}^M(X,p)^{\bullet}$:
$\underline{Gers}^M(X,p)^{\bullet}(U)=Gers^M(U,p)^{\bullet}$ for an
open subset $U\subset X$, where
$$
Gers^M(U,p)^l=\sum_{\eta\in X^{(l)}} K_{p-l}^M(k(\eta))
$$
and $K^M$ denotes the Milnor $K$-groups. Combining the canonical
homomorphism from Milnor $K$-groups to Quillen $K$-groups of fields and
the fact that the Gersten complex is a resolution of the sheaf $\K_p$, we get a morphism
$$
R\Gamma(X,Z^p(X)^{\bullet})\to R\Gamma(X,\K_p[-p]).
$$
Moreover, it follows from \cite{NS} and \cite{Tot} that this morphism
commutes with the multiplication morphisms. By
Proposition~\ref{prop-Chowbiext}, we get the needed result.
\end{proof}

\begin{proof}[Proof 2]
Let us give an explicit proof in the case
when the ground field is infinite and perfect.
We use adelic complexes (see
Section~\ref{subsection-adelic} and \cite{Gor07}).
The multiplicative structure on the adelic complexes defines the
morphism of complexes
\mbox{$\phi:\A(X,p)^{\bullet}\otimes\A(X,q)^{\bullet}\to k^*[-d]$},
which agrees with the natural morphism
\mbox{$\Hom_{D^b(\Abc)}(R\Gamma(X,\K_p)\otimes^L_{\Z}R\Gamma(X,\K_q),k^*[-d])$}
that defines the biextension $P_{KC}$.

By $d$ denote the differential in the adelic complexes
$\A(X,n)^{\bullet}$. In notations from
Section~\ref{subsection-complexes}, let $T\subset
\Ker(d^p)'\times\Ker(d^{q})'$ be the bisubgroup that consists of all
pairs $(\sum_i[Z_i],\sum_j[W_j])$ with $[Z_i]\in \A(X,\K_p)^p$,
$[W_j]\in\A(X,q)^{q}$ such that for all $i$, $j$, we have:
\begin{itemize}
\item[(i)] the $K$-adeles $[Z_i]\in \A(X,\K_p)^p$ and
$[W_j]\in\A(X,q)^{q}$ are good cocycles with respect to some
patching systems $\{(Z_i)^{1,2}_r\}$ and $\{(W_j)^{1,2}_s\}$ for
cycles $Z_i\in Z^p(X)'$ and $W_j\in Z^q(X)'$ on $X$, respectively;
\item[(ii)] the patching systems
$\{(Z_i)^{1,2}_r\}$ and $\{(W_j)^{1,2}_s\}$ satisfy the condition
$(i)$ from Lemma~\ref{lemma-patchingsyst} and the patching system
$\{(W_j)^{1,2}_s\}$ satisfies the condition $(ii)$ from
Lemma~\ref{lemma-patchingsyst} with respect to the subvariety
$|Z_i|\subset X$;
\item[(iii)] the support $|Z_i|$ meets the support $|W_j|$
properly.
\end{itemize}

In particular, for all $j$, we have
$\codim_{Z}(Z\cap((W_j)_s^1\cap (W_j)_s^2))\ge s+1$ for all $s$,
$1\le s\le q-1$, where $Z=|\sum_i Z_i|\in Z^p(X)$.

Combining the classical moving lemma,
Lemma \ref{lemma-patchingsyst}, and Lemma \ref{lemma-adelicgoodcocycles}, we see that
the natural map $T\to H^p(X,\K_p)'\times H^{q}(X,\K_{q})'=
CH^p(X)'\times CH^{q}(X)'$ is surjective.

Suppose that $(\sum_i[Z_i],\sum_j[W_j])\in T$ and the class of the
cycle $Z=\sum_i Z_i$ in $H^p(X,\K_p)=CH^p(X)$ is trivial. By
Lemma~\ref{lemma-movingK1} and Corollary \ref{corol-movingK1codim},
there exists a \mbox{$K_1$-chain} $\{f_{\eta}\}\in G^{p-1}(X,1)$
such that $\div(\{f_{\eta}\})=Z$, the support
$Y=\Supp(\{f_{\eta}\})$ meets $W=\sum_j W_j\in Z^{q}(X)$ properly
and for all $j$, we have $\codim_{Y}(Y\cap((W_j)_s^1\cap
(W_j)_s^2)))\ge s+1$ for all $s$, $1\le s\le q-1$. It follows that
for all $j$, the patching system $\{(W_j)^{1,2}_s\}$ satisfies the
condition $(ii)$ from Lemma \ref{lemma-patchingsyst} with respect to
the subvariety $Y\subset X$. Combining Lemma
\ref{lemma-patchingsyst}, Lemma \ref{lemma-adelicboundary}, and
Lemma \ref{lemma-adelicproduct}, we see that there exists a
$K$-adele $[\{f_{\eta}\}]\in\A(X,\K_p)^{p-1}$ such that
$d^{p-1}[\{f_{\eta}\}]=\sum_i[Z_i]$ and
$\pi_*(m([\{f_{\eta}\}]\otimes\sum_j[W_j]))=
(-1)^{pq}\prod_{\eta}f_{\eta}(\overline{\eta}\cap W)$.

Suppose that $(\sum_i[Z_i],\sum_j[W_j])\in T$ and the class of the
cycle $W=\sum_j W_j$ in $H^{q}(X,\K_{q})=CH^{q}(X)$ is trivial. By
Lemma \ref{lemma-movingK1}, there exists a $K_1$-chain
$\{g_{\xi}\}\in G^{p-1}(X,1)$ such that $\div(\{g_{\xi}\})=Z$ and
the support $Y=\Supp(\{g_{\xi}\})$ meets $Z=\sum_i Z_i\in Z^{p}(X)$
properly. Combining Lemma \ref{lemma-patchingsyst}, Lemma
\ref{lemma-adelicboundary}, and Lemma \ref{lemma-adelicproduct}, we
see that there exists a $K$-adele
$[\{g_{\xi}\}]\in\A(X,\K_{q})^{q-1}$ such that
$d^{q-1}[\{g_{\xi}\}]=\sum_j[W_j]$ and $(-1)^p
\pi_*(m([Z]\otimes[\{g_{\xi}\}]))=
(-1)^{pq}\prod_{\xi}g_{\xi}(Z\cap\overline{\xi})$.

Therefore by Remark \ref{rmk-biextcoinc} applied to $T\subset \Ker(d^p)'\times
\Ker(d^q)'$ and $\psi$ induced by $\phi=\pi_*\circ m$, we get the needed result.

Therefore the needed statement follows from
Remark~\ref{rmk-biextcoinc} applied to the bisubgroup $T\subset
\Ker(d^p)'\times \Ker(d^q)'$ and the bigger bisubgroup in
$\Ker(d^p)'\times \Ker(d^q)'$ together with the bilinear map defined
by the morphism of complexes $\phi$ as shown in
Section~\ref{subsection-complexes}.
\end{proof}

\subsection{Determinant of cohomology construction}\label{subsection-determbiext}

In Section~\ref{subsection-determ} we defined
a ``determinant of cohomology'' distributive functor
$$
\langle\cdot,\cdot\rangle:V\M_X\times V\M_X\to V\M_k.
$$
Our goal is to descend this distributive functor to a biextension of Chow groups.
The strategy is as follows. First,
we define a filtration $C^pV\M_X$ on the category $V\M_X$, which is a kind
of a filtration by codimension of support. The successive quotients of
this filtration are isomorphic to certain Picard categories $\widetilde{CH}^p_X$
that are related to Chow groups. Then we show a homotopy invariance of
the categories $C^pV\M_X$ and construct a specialization map for them.
As usual, this allows to defined a contravariant structure on the categories $C^pV\M_X$
by using deformation to the normal cone. Next, the exterior product between
the categories $C^pV\M_X$ together with the pull-back along diagonal allows one
to define for a smooth variety $X$ a collection of distributive functors
$$
C^pV\M_X\times C^qV\M_X\to C^{p+q}V\M_X,
$$
compatible with the distributive functor
$$
V\M_X\times V\M_X\to V\M_X
$$
defined by the derived tensor product of coherent sheaves.
Since all constructions for the categories $C^pV\M_X$
are compatible for different $p$, we get the analogous constructions for the categories
$\widetilde{CH}^p_X$ (contravariancy and a distributive functor). Taking
the push-forward functor $C^{d+1}V\M_X\to V\M_X\to V\M_k$ for a smooth projective variety $X$,
we get a distributive functor
$$
\langle\cdot,\cdot\rangle_{pq}:\widetilde{CH}^p_X\times \widetilde{CH}^q_X\to V\M_k
$$
for $p+q=d+1$. Finally, applying Lemma~\ref{lemma-Deligne}, we get a biextension $P_{DC}$ of
$(CH^p(X)',CH^q(X)')$ by $k^*$. The fiber of this biextension at the classes of algebraic
cycles $Z=\sum_i m_iZ_i$ and $W=\sum_j n_jW_j$ is canonically isomorphic to the $k^*$-torsor
$$
\mbox{$\det(\sum_{i,j}m_in_jR\Gamma(X,\OO_{Z_i}\otimes^L_{\OO_X}\OO_{W_j}))\backslash\{0\}.$}
$$
Also, we prove that the biextension $P_{DC}$ is canonically isomorphic
to the biextension $P_E$ constructed in Section~\ref{subsection-acycles}.

Let us follow this plan.
For any variety $X$, and $p\ge 0$ denote by $\M^p_X$ the exact category
of sheaves on $X$ whose support codimension
is at least $p$. By definition, put $\M^{p}_X=\M_X$ for $p<0$.
For $p\ge q$, there are natural functors $V\M^p_X\to V\M_X$
and $V\M^p_X\to V\M^{q}_X$ .

\begin{defin}\label{defin-filtr}
For $p\ge 0$, let $C^pV\M_X$ be the following Picard category:
objects in $C^pV\M_X$ are objects in the category $V\M^p_X$
and morphisms are defined by the formula
$$
\Hom_{C^pV\M_X}(L,M)
=\mathrm{Im}(\Hom_{V\M_X^{p-1}}(L,M)\to
\Hom_{V\M_X^{p-2}}(L,M))
$$
(more precisely, we consider images of objects $L$ and $M$ with respect
to the corresponding functors from $V\M^p_X$);
a monoidal structure on $C^pV\M_X$ is naturally defined by monoidal structures
on the categories $V\M^{*}_X$.
\end{defin}

By definition, we have $C^0V\M_X=V\M_X$. Note that for $p> d+1$, $C^pV\M_X=0$  and
$C^{d+1}V\M_X$ consists of one object $0$ whose automorphisms group is equal to
$\mathrm{Im}(K_1(\M^d_X)\to K_1(\M^{d-1}_X))$ (the last group equals $H^d(X,\K_{d+1})$
provided that $X$ is smooth). For $p\ge q\ge 0$, there are natural functors
$C^pV\M_X\to C^qV\M_X$. It is clear that
$$
\pi_0(C^pV\M_X)=\mathrm{Im}(K_0(\M^p_X)\to K_0(\M_X^{p-1})),
$$
$$
\pi_1(C^pV\M_X)=\mathrm{Im}(K_1(\M^{p-1}_X)\to K_1(\M_X^{p-2})).
$$

The next definition is the same as the definition given in \cite[Section 2]{Fra}.

\begin{defin}
For $p\ge 1$, let $\widetilde{CH}^p_X$ be the following Picard category:
objects in $\widetilde{CH}^p_X$ are elements in the group $Z^p(X)$
and morphisms are defined by the formula
$$
\Hom_{\widetilde{CH}^p_X}(Z,W)
=\{f\in G^{p-1}(X,1)|\div(f)=W-Z\}/\Tame,
$$
where $\Tame$ denotes the $K_2$-equivalence on $K_1$-chains, i.e.,
the equivalence defined by the homomorphism
$$
\Tame:G^{p-2}(X,2)\to G^{p-1}(X,1);
$$
a monoidal structure on $\widetilde{CH}^p_X$
is defined by taking sums of algebraic cycles.
\end{defin}

It is clear that for $p\ge 0$, we have
$$
\pi_0(\widetilde{CH}^p_X)=CH^p(X),
$$
$$
\pi_1(\widetilde{CH}^p_X)=H^{p-1}(Gers(X,p)^{\bullet})=\Ker(\div)/\mathrm{Im}(\Tame).
$$

\begin{rmk}
In notations from Remark \ref{rmk-Picardbiext}, we have
$$
\widetilde{CH}^p_X=\Picard(\tau_{\ge(p-1)} Gers(X,p)^{\bullet}).
$$
\end{rmk}

\begin{lemma}\label{lemma-filtrcomp}
For any $p\ge 0$, there is a canonical equivalence of Picard categories
$$
C^pV\M_X/C^{p-1}V\M_X\to \widetilde{CH}^p_X.
$$
\end{lemma}
\begin{proof}
First, let us construct a functor $F^p:C^pV\M_X\to\widetilde{CH}^p_X$.
By Example~\ref{examp-localiz} (ii),
for any $p\ge 0$, there is an equivalence of Picard categories
$H^p:V\M^p_X/V\M^{p-1}_X\to \oplus_{\eta\in X^{(p)}}V\M_{k(\eta)}$.
In notation from Section~\ref{subsection-determ},
this defines a functor $\widetilde{H}^p:V\M^p_X\to \Picard(Z^p(X))$. We put
$F^p(L)=\widetilde{H}^p(L)\in Z^p(X)$
for any object $L$ in $C^pV\M_X$. Let $f:L\to M$ be a morphism
in $C^pV\M_X$ and let $\widetilde{f}:L\to M$ be a corresponding morphism in the category
$V\M^{p-1}_X$ (note that $\widetilde{f}$ is not uniquely defined).
There are canonical isomorphisms $H^{p-1}(L)\to 0$, $H^{p-1}(M)\to 0$, therefore
$H^{p-1}(\widetilde{f})$ defines a canonical element in the group
$\pi_1(\oplus_{\eta\in X^{(p-1)}}V\M_{k(\eta)})=G^{p-1}(X,1)$ that we denote again by
$H^{p-1}(\widetilde{f})$. It is easy to check that
$\div(H^{p-1}(\widetilde{f}))=F^p(M)-F^p(L)$. We put $F^p(f)=[H^{p-1}(\widetilde{f})]$,
where brackets denote the class of a $K_1$-chain modulo $K_2$-equivalence.
The exact sequence
$$
G^{p-2}(X,2)\to K_1(\M^{p-1}_X)\to K_1(\M^{p-2}_X)
$$
shows that $[F^{p-1}(\widetilde{f})]$ does not depend on the choice of
$\widetilde{f}$.

Next, the exact sequences
$$
K_0(\M^{p+1}_X)\to K_0(\M^p_X)\to Z^p(X),
$$
$$
K_1(\M^{p}_X)\to K_1(\M^{p-1}_X)\to G^{p-1}(X,1)
$$
show that the composition $C^{p+1}V\M_X\to C^pV\M_X\to \widetilde{CH}^p_X$
is canonically trivial. By Lemma~\ref{lemma-localiz}(iii), we get a well-defined
functor
$$
C^pV\M_X/C^{p-1}V\M_X\to \widetilde{CH}^p_X.
$$
Finally, combining Lemma~\ref{lemma-localiz}(ii) with the explicit description of $\pi_i$
for all involved categories, we see that the last functor is an isomorphism on the groups
$\pi_i$, $i=0,1$.
This gives the needed result.
\end{proof}

\begin{rmk}\label{rmk-inverseChow}
Let us construct explicitly an inverse functor $(F^p)^{-1}$
to the equivalence $F^p:C^pV\M_X/C^{p+1}V\M_X\to \widetilde{CH}^p_X$.
For each cycle $Z=\sum_i m_i Z_i$, we put
$(F^p)^{-1}(Z)=\sum_i m_i\gamma (\OO_{Z_i})$, where we choose
an order on the set of summands in the last expression.
For a codimension $p-1$ irreducible subvariety $Y\subset X$
and a rational function $f\in k(Y)^*$, let $\widetilde{Y}$ be the
closure of the graph of the rational function
$f:Y\dasharrow\P^1$ and let $\pi:\widetilde{Y}\to
Y\hookrightarrow X$ be the natural map. Denote by $D_0$ and
$D_{\infty}$ Cartier divisors of zeroes and poles of $f$ on
$\widetilde{Y}$, respectively. Denote by $Z_0$ and $Z_{\infty}$ the positive and the negative
part of the cycle $\div(f)$ on $X$, respectively. The exact sequences
$$
0\to \OO_{\widetilde{Y}}(-D_{\infty})\to\OO_{\widetilde{Y}}\to \OO_{D_{\infty}}\to 0,
$$
$$
0\to \OO_{\widetilde{Y}}(-D_{0})\to\OO_{\widetilde{Y}}\to \OO_{D_{0}}\to 0,
$$
and the isomorphism $f:\OO_{\widetilde{Y}}(-D_{\infty})\to
\OO_{\widetilde{Y}}(-D_{0})$ define an element in the set
$$
\Hom_{V\M^{p-1}_X}(\sum_i (-1)^i \gamma(R^i\pi_*\OO_{D_{\infty}},\sum_i (-1)^i
\gamma(R^i\pi_*\OO_{D_0})),
$$
that in turn defines an element
$$
(F^p)^{-1}(f_{\eta})\in\Hom_{V\M^{p-1}_X}((F^p)^{-1}(Z_{\infty}),(F^p)^{-1}(Z_0)+L)
$$
up to morphisms in the category $V\M^p_X$, where $L$ is a well-defined object in $V\M^{p+1}_X$.
For each morphism in the category $\widetilde{CH}^p_X$, we choose its representation
as the composition of morphisms defined by $f$ for some codimension $p-1$ irreducible subvarieties
$Y\subset X$. This allows to extend $(F^p)^{-1}$ to all morphisms in the category $\widetilde{CH}^p_X$.
\end{rmk}

Now let us describe some functorial properties of the categories $C^pV\M_X$.

\begin{defin}
For a variety $X$, let $\{\Pc^p_X\}$, $p\ge 0$,
be the collection of categories $\{V\M^p_X\}$ or $\{C^p\M_X\}$
(the same choice for all varieties), $G_X^{pq}:\Pc^p\to\Pc^q$ be the natural functors,
and let $S$ and $T$ be two varieties.
A {\it collection of compatible functors} from $\{\Pc^p_S\}$ to $\{\Pc^p_T\}$
is a pair $(\{F^p\},\{\Psi^{pq}\})$, where $F^p:\Pc^p_S\to \Pc^p_T$, $p\ge 0$, are
symmetric monoidal functors and $\Psi^{pq}:G_T^{pq}\circ F^p\to F^q\circ G_S^{pq}$
are isomorphisms in the category $\Funt(\Pc^p_S,\Pc^q_T)$
such that the following condition is satisfied: for all $p\ge q\ge r\ge 0$, we have
$G^{pq}_S(\Psi^{qr})\circ G^{qr}_T(\Psi^{pq})=\Psi^{pr}$.
\end{defin}

The proof of the next result is straightforward.

\begin{lemma}\label{lemma-filtr}
For varieties $S$ and $T$,
a collection of compatible functors from $\{V\M^p_S\}$ to $\{V\M^p_T\}$
defines in a canonical way a collection of compatible functors from
$\{C^p\M_S\}$ to $\{C^p\M_S\}$.
\end{lemma}

\begin{examp}\label{examp-flatcontrav}
Let $f:S\to T$ be a flat morphism of schemes; then there is a collection
of compatible functors $f^*:V\M^p_T\to V\M^p_S$.
By Lemma~\ref{lemma-filtr}, this gives a collection of compatible functors
$f^*:C^p\M_T\to C^p\M_S$.
\end{examp}

Now let us prove homotopy invariance of the categories $C^pV\M_X$.
Note that there is no homotopy invariance for the categories $V\M^p_X$.
We use the following straightforward result on truncations of spectral sequences:

\begin{lemma}\label{lemma-spectr}
\hspace{0cm}
\begin{itemize}
\item[(i)]
Consider a spectral sequence $E^{ij}_r$, $r\ge 1$, a natural number $s$,
and an integer $l\in\Z$. Then there is a unique spectral sequence
$(t^l_s E)^{ij}_r$, $r\ge s$, such that $(t^l_s E)^{ij}_s=E^{ij}_s$ if $i\ge l$,
$(t^l_s E)^{ij}_s=0$ if $i<l$, and there is a morphism of spectral sequences
$(t^l_s E)^{ij}_r\to E^{ij}_r$, $r\ge s$.
\item[(ii)]
In the above notations, let $s_1,s_2$ be two natural numbers; then we have
$(t^{s_1}_l E)^{ij}_r=(t^{s_2}_l E)^{ij}_r$ for all $r\ge s=\mathrm{max}\{s_1,s_2\}$ and
$i\ge l+s-1$.
\end{itemize}
\end{lemma}

\begin{lemma}\label{lemma-hominv}
Consider a vector bundle $f:N\to S$; then the natural functors
$f^*:C^pV\M_S\to C^pV\M_N$
are equivalencies of Picard categories.
\end{lemma}
\begin{proof}
We follow the proof of \cite[Lemma 81]{Gil}. It is enough to show that $f^*$
induces isomorphisms on $\pi_0$ and $\pi_1$, i.e., it is enough to show that the natural
homomorphisms
$$
f^*:\mathrm{Im}(K_n(\M^p_S)\to K_n(\M^{p-1}_S))\to\mathrm{Im}(K_n(\M^p_N)\to K_n(\M^{p-1}_N))
$$
are isomorphisms for all $p\ge 0$, $n\ge 0$.

For any variety $T$,
consider Quillen spectral sequence $E^{ij}_r(T)$, $r\ge 1$, that converges to
$K_{-i-j}(\M_T)$. In notations from Lemma~\ref{lemma-spectr},
the filtration of abelian categories $\M^{p-1}_T\supset \M^p_T\supset\ldots$
defines the spectral sequence that is equal to the shift of the truncated Quillen
spectral sequence $(t^{p-1}_1E)^{ij}_{r}(T)$, $r\ge 1$ and converges to
$K_{-i-j}(\M^{p-1}_T)$. By Lemma~\ref{lemma-spectr},
$(t^{p-1}_1E)^{ij}_{r}(T)=(t^{p-1}_2E)^{ij}_{r}(T)$ for $i\ge p$.
Therefore, there is a spectral sequence $(t^{p-1}_2E)^{ij}_{r}(T)$, $r\ge 2$
that converges to $\mathrm{Im}(K_{-i-j}(\M^p_T)\to K_{-i-j}(\M^{p-1}_T))$.
Explicitly, we have $(t^{p-1}_2E)^{ij}_{2}(T)=H^i(Gers(T,-j)^{\bullet})$.

The morphism $f$ defines the morphism of spectral sequences
$f^*:(t^{p-1}_2E)^{ij}_{r}(S)\to (t^{p-1}_2E)^{ij}_{r}(N)$, $r\ge 2$.
Moreover, the homomorphism $f^*$ is an isomorphism for $r=2$ (see op.cit.), hence
$f^*$ is an isomorphism for all $r\ge 2$. This gives the needed result.
\end{proof}

\begin{corol}\label{corol-inversehomot}
Taking inverse functors to the equivalences $f^*:C^pV\M_V\to C^pV\M_N$,
we get a collection of compatible functors $(f^*)^{-1}$ from $C^pV\M_N$ to $C^pV\M_S$.
\end{corol}

Next we construct a specialization map for the categories $C^pV\M_X$.

\begin{lemma}\label{lemma-special}
Let $j:D\subset Y$ be a subscheme on a scheme given as a subscheme by the
equation $f=0$, where
$f$ is a regular function on $Y$, and let $U=Y\backslash D$. Then there exists a
collection of compatible functors $Sp^p:V\M^p_U\to V\M^p_D$ such that
the composition $V\M_Y\to V\M_U\stackrel{Sp^0}\longrightarrow V\M_D$
is canonically isomorphic to $j^*$ in the category of symmetric monoidal functors.
\end{lemma}
\begin{proof}
The composition $j^*\circ\gamma_Y:(\M_Y)_{iso}\to V\M_D$ is canonically isomorphic
to the functor $\F\mapsto \gamma_D(\G_0)-\gamma_D(\G_{-1})$, where $\G_i$
are cohomology sheaves of the complex $\F\stackrel{f}\longrightarrow\F$
placed in degrees $-1$ and $0$. It follows that the composition
$j^*\circ j_*\circ\gamma_D:(\M_D)_{iso}\to V\M_D$ is canonically
isomorphic to $0$ in the category $\mathrm{Det}(\M_D,V\M_D)$. By Lemma~\ref{lemma-univ},
this defines a canonical isomorphism $j^*\circ j_*\to 0$ in the category $\Funt(V\M_D,V\M_D)$.
Combining Lemma~\ref{lemma-localiz} (iii) and Example~\ref{examp-localiz} (ii),
we get the functor $\mathrm{Sp}:V\M_U\to V\M_D$, since there is a natural equivalence
of abelian categories $\M_Y/\M_D\to \M_U$.

More explicitly, the functor $\mathrm{sp}:V\M_U\to V\M_D$ corresponds via Lemma~\ref{lemma-univ} to the following determinant functor:
$\F\mapsto \gamma_D(j^*\widetilde{\F})$, where $\F$ is a coherent sheaf on $U$ and
$\widetilde{\F}$ is coherent sheaf on $Y$ that restricts to $\F$ (one can easily check
that two different collections of choices of $\widetilde{\F}$ for all $\F$ define canonically
isomorphic determinant functors).

Take a closed subscheme $Z\subset U$ such that each component of $Z$
has codimension at most $p$ in $U$. Consider the closed subscheme $Z_D=\overline{Z}\times _X D$ in
$D$, where $\overline{Z}$ is the closure of $Z$ in $X$.
We get a functor $Sp_Z:V\M_Z\to V\M_{Z_D}\to V\M^p_D$. Given a diagram
$Z\stackrel{i}\hookrightarrow Z'\subset U$,
we have a canonical isomorphism of symmetric monoidal functors
$Sp_{Z'}\circ i_*\to Sp_Z$. Taking the limits over closed
subschemes $Z\subset U$, we get the needed collection of compatible
functors $Sp^p:V\M^p_U\to V\M^p_D$, $p\ge 0$.
\end{proof}

\begin{corol}\label{corol-spec}
Combining Lemma~\ref{lemma-special} and Lemma~\ref{lemma-filtr},
we get a collection of compatible functors $Sp^p:C^pV\M_U\to C^pV\M_D$ in notations
from Lemma~\ref{lemma-special}.
\end{corol}

Now we show contravariancy for the categories $C^pV\M_X$.
Let $i:S\subset T$ be a regular embedding of varieties;
then there is a symmetric monoidal functor
$$
i^*:V\M_T\cong V\M'_T\to V\M_S
$$
where $\M'_T$ is a subcategory in $\M_T$ of $\OO_S$-flat coherent sheaves.
The composition of $\gamma_T:(\M_T)_{iso}\to V\M_T$
with $i^*$ is canonically isomorphic to a functor that sends a coherent sheaf
$\F$ on $T$ to the object $\sum_i(-1)^i\gamma(\mathcal{T}or^{\OO_T}_i(\F,\OO_S))$.
By Lemma~\ref{lemma-univ}, the last condition defines the functor $i^*$
up to a unique isomorphism.

%\begin{prop}\label{prop-restr}
%Let $i:S\subset T$ be a regular embedding of varieties;
%then for all $p\ge 0$, the functor $i^*$ sends the subcategory $F^p$
%\end{prop}

\begin{prop}\label{prop-restr}
For each $p\ge 0$, there exists a
collection of compatible functors
$i_p^*:C^pV\M_T\to C^pV\M_S$ such that $i^*_0=i^*$.
\end{prop}
\begin{proof}
We use deformation to the normal cone (see \cite{Ful}). Let $M$ be the blow-up
of $S\times\{\infty\}$ in $T\times\P^1$ and let $M^0=M\backslash\P(N)$, where
$N$ denotes the normal bundle to $S$ in $T$. Then we have
$M^0|_{\Ab^1}\cong T\times \Ab^1$ and $M^0|_{\{\infty\}}=N$,
where $\Ab^1=\P^1\backslash\{\infty\}$. Combining Corollary~\ref{corol-inversehomot} and
Corollary~\ref{corol-spec}, we get a compatible collection of functors $i^*_p$
as the composition
$$
C^pV\M_T\stackrel{pr_T^*}\longrightarrow
C^pV\M_{T\times\Ab^1}\stackrel{Sp^p}\longrightarrow
C^pV\M_{N}\stackrel{(pr_S^*)^{-1}}\longrightarrow C^pV\M_S,
$$
where $pr_T:T\times\Ab^1\to T$, $pr_S:N\to S$ are natural projections, and
$(pr_S^*)^{-1}$ is an inverse to the equivalence $pr_S^*:C^pV\M_S\to C^pV\M_N$.
\end{proof}

\begin{rmk}\label{rmk-properint}
Let $\M^p_{T,S}$ be a full subcategory in $\M^p_T$ whose objects are
coherent sheaves $\F$ from $\M^p_T$ whose support meets $S$ properly.
Then the composition of the natural functor $(\M^p_{T,S})_{iso}\to C^pV\M_T$ with
$i_p^*$ is canonically isomorphic to the functor that sends $\F$ in $\M^p_{T,S}$
to the object $\sum_i(-1)^i\gamma(\mathcal{T}or^{\OO_T}_i(\F,\OO_S))$.
\end{rmk}

%such that the composition of the natural functor $(\M^p_X)_{iso}\to F^p\M_X$ with $i_p^*$
%is equal to a determinant functor that sends $\F$ to the object
%$\sum_i(-1)^i\gamma(\mathcal{T}or^{\OO_T}_i(\F,\OO_S))$ for a coherent sheaf $\F$ from
%$\M^p_X$ whose support meets $S$ properly.
%The above conditions define the collection
%of functors $i^*_p$ uniquely up to a canonical isomorphism.

\begin{rmk}\label{rmk-contrav}
Applying Proposition~\ref{prop-restr} to the embedding of the graph of
a morphism of varieties $f:X\to Y$ with smooth $Y$,
we get a collection of compatible functors
$f^*:C^pV\M_Y\to C^p\M_X$. Combining Lemma~\ref{lemma-filtrcomp}
and Lemma~\ref{lemma-localiz} (iii),
we get a collection of symmetric monoidal pull-back functors
$f^*:\widetilde{CH}^p_Y\to\widetilde{CH}^p_X$.
\end{rmk}

\begin{rmk}
The pull-back functors for the categories $\widetilde{CH}^p_X$ are also
constructed by J.\,Franke in \cite{Fra} by different methods: the categories
$\widetilde{CH}^p_X$ are interpreted as equivalent categories to categories
of torsors under certain sheaves. Also, a biextension of Chow groups
is constructed by the same procedure as below. It is expected that
the pull-back functor constructed in op.cit. is canonically isomorphic
to the pull-back functor constructed in Remark~\ref{rmk-contrav}.
By Proposition~\ref{prop-DCbiext} below,
this would imply that the biextension constructed in op.cit. is canonically isomorphic
to the biextension $P_E$.
\end{rmk}

\begin{rmk}
By Lemma~\ref{lemma-adelicquasiis}, for a smooth variety $X$ over an infinite perfect field $k$,
there is a canonical equivalence
$$
\Picard(\tau_{\ge(p-1)}\tau_{\le p}\A(X,\K_p)^{\bullet})\to \widetilde{CH}^p_X.
$$
Moreover, the Picard categories on the left hand side are canonically contravariant.
This gives another way to define a contravariant structure on the Chow categories
(note that the Godement resolution for a sheaf $\K_p$ is not enough to do this).
\end{rmk}

For three varieties $R,S,T$, a {\it collection of compatible distributive functors}
from $C^pV\M_R\times C^qV\M_S$ to $C^{p+q}V\M_T$, $p,q\ge 0$
is defined analogously to the collection of compatible symmetric tensor functors.

\begin{corol}\label{corol-product}
Let $X$ be a smooth variety.
There exists a collection of compatible distributive functors
$$
C^pV\M_X\times C^qV\M_X\to C^{p+q}V\M_X
$$
such that for $p=q=0$, the composition of the corresponding functor with
the functor $(\M_X)_{iso}\times(\M_X)_{iso}\to V\M_X\times V\M_X$ is canonically
isomorphic to the functor
$(\F,\G)\mapsto \sum_i(-1)^i\gamma(\mathcal{T}or^{\OO_X}_i(\F,\G))$.
\end{corol}
\begin{proof}
The exterior product of sheaves defines an collection of compatible distributive functors
$$
C^pV\M_X\times C^qV\M_X\to C^{p+q}\M_{X\times X}.
$$
Applying Proposition~\ref{prop-restr} to the diagonal embedding
$X\subset X\times X$, we get the needed statement.
\end{proof}

For algebraic cycles $Z=\sum_i m_i Z_i$, $W=\sum_j n_jW_j$, we put
$$
\langle Z,W\rangle=
\mbox{$\det(\sum_{i,j}m_in_jR\Gamma(X,\OO_{Z_i}\otimes^L_{\OO_X}\OO_{W_j}))\backslash\{0\}.$})
$$

\begin{prop}\label{prop-detcohombiext}
Let $X$ be a smooth projective variety of dimension $d$. Then for $p+q=d+1$,
there is a biextension $P_{DC}$ of $(CH^p(X)',CH^q(X)')$ such that the fiber
$P_{DC}|_{(Z,W)}$ is canonically isomorphic to the $k^*$-torsor $\langle Z,W\rangle$.
\end{prop}
\begin{proof}
By Lemma~\ref{lemma-localiz} (iv), the collection of compatible distributive functors from
Corollary~\ref{corol-product} defines a collection of distributive functors
$$
\widetilde{CH}^p_X\times\widetilde{CH}^q_X\to
\widetilde{CH}^{p+q}_X.
$$
for all $p,q\ge 0$. If $p+q=d+1$, then we get
a distributive functor
$$
\langle\cdot,\cdot\rangle_{pq}:\widetilde{CH}^p_X\times\widetilde{CH}^q_X\to
\widetilde{CH}^{d+1}_X=C^{d+1}V\M_X\to V\M_X\to V\M_k.
$$
Moreover, by Remark~\ref{rmk-inverseChow},
we have $\langle Z,W\rangle_{pq}=(0,\langle Z,W\rangle)$.

Consider full Picard subcategories
$(\widetilde{CH}^p_X)'\subset \widetilde{CH}^p_X$ whose objects are elements $Z\in Z^p(X)'$
and the restriction of the distributive functor $\langle\cdot,\cdot\rangle$
to the product $(\widetilde{CH}^p_X)'\times (\widetilde{CH}^q_X)'$.
Applying Lemma~\ref{lemma-Deligne}, we get the next result.
\end{proof}

Finally, let us compare the biextension $P_{DC}$ with the biextension $P_E$.
With this aim we will need the following result.

\begin{lemma}\label{lemma-computation}
Let $X$ be a smooth projective variety of dimension $d$, $p+q=d+1$,
let $Y\subset X$ be a codimension $p-1$ irreducible subvariety, $f\in k(Y)^*$,
and let $W\subset X$ be a codimension $q$ irreducible subvariety.
Suppose that $Y$ meets $W$ properly and that $|\div(f)|\cap
W=\emptyset$. Then $f$ defines a isomorphism $0\to \div(f)$ in $\widetilde{CH}^p_X$,
hence a morphism $\sigma(f):k^*=\langle 0,W \rangle\to \langle \div(f),W\rangle=k^*$
such that $\sigma(f)=f(Y\cap W)$.
\end{lemma}
\begin{proof}
By Remark~\ref{rmk-inverseChow} and in its notations,
the action of $f$ in question is induced by the exact sequences
$$
0\to \OO_{\widetilde{Y}}(-D_{\infty})\to\OO_{\widetilde{Y}}\to \OO_{D_{\infty}}\to 0,
$$
$$
0\to \OO_{\widetilde{Y}}(-D_{0})\to\OO_{\widetilde{Y}}\to \OO_{D_{0}}\to 0,
$$
and the isomorphism $f:\OO_{\widetilde{Y}}(-D_{\infty})\to
\OO_{\widetilde{Y}}(-D_{0})$. Moreover, $\pi:\widetilde{Y}\to Y$ is an isomorphism outside of
$D_0\cup D_{\infty}$, so $\sigma(f)$ is induced by the automorphism of the complex
$\OO_Y\otimes_{\OO_X}^L\OO_W$ given by multiplication on $f$. Applying
the functor $\det R\Gamma(-)$, we get $f(Y\cap W)$.
\end{proof}

\begin{prop}\label{prop-DCbiext}
The biextension $P_{DC}$ is canonically isomorphic to the biextension $P_E$.
\end{prop}
\begin{proof}
Let $T\subset Z^p(X)'\times Z^q(X)'$
be the bisubgroup that consists of pair $(Z,W)$ such that
$|Z|\cap |W|=\emptyset$. Suppose that $(Z,W)$, $(Z',W)$ are in $T$ and that $Z$
is rationally equivalent to $Z'$. By Lemma~\ref{lemma-movingK1}, we may suppose
that there is a $K_1$-chain $\{f_{\eta}\}\in G^{p-1}_{|W|}(X,1)$ such that
$\div(\{f_{\eta}\})=Z'-Z$. By Lemma
\ref{lemma-computation}, we immediately get the needed result.
\end{proof}

Combining Proposition \ref{prop-DCbiext}, Proposition
\ref{prop-AbelJacobi}, and Remark \ref{rmk-explicPoinccycles}, we obtain
the following statement.

\begin{corol}\label{corol-canon-torsors}
Suppose that $X$ is a complex smooth projective variety of dimension
$d$, $Z\in Z^p(X)_{\hom}$, $W\in Z^q(X)_{\hom}$, $p+q=d+1$; then
there is a canonical isomorphism of the following $\C^*$-torsors:
\begin{itemize}
\item
$\det R\Gamma(X,\OO_Z\otimes^L_{\OO_X}\OO_W)\backslash\{0\}$;
\item
the set of isomorphism classes of all integral mixed Hodge
structures $V$ whose weight graded quotients are identified with
$\Z(0)$, $H^{2p-1}(X)(p)$, $\Z(1)$ and such that
$[V/W_{-2}V]=[E_Z]\in \Ext^1_{\H}(\Z(0),H^{2p-1}(X)(p))$,
$[W_{-1}V]=[E_{W}^{\vee}]\in \Ext^1_{\H}(H^{2p-1}(X)(p),\Z(1))$.
\end{itemize}
\end{corol}

\begin{quest}
Does there exist a direct proof of Corollary
\ref{corol-canon-torsors}?
\end{quest}

The present approach is a higher-dimensional generalization for the description
of the Poincar\'e biextension on a smooth projective curve $C$ in terms of
determinant of cohomology of invertible sheaves, suggested by
P.\,Deligne in \cite{Del}. Let us briefly recall this construction
for the biextension $\Pc$ of $(\Pic^0(C),\Pic^0(C))$ by $k^*$. For any degree zero
invertible sheaves $\Lc$ and $\M$ on $C$ there is an equality
$$
\Pc_{([\Lc],[\M])}=\langle\Lc-\OO_C,\M-\OO_C\rangle= \mbox{$\det
R\Gamma(C,\Lc\otimes_{\OO_C}\M-\Lc-\M+\OO_C))\backslash\{0\}$},
$$
where $[\cdot]$ denotes the isomorphism class of an invertible sheaf.
Since $\chi(C,\Lc\otimes_{\OO_C}\M)=\chi(C,\Lc)=\chi(C,\M)=1-g$,
this \mbox{$k^*$-torsor} is well defined on
$\Pic^0(C)\times\Pic^0(C)$. Consider a homomorphism
\mbox{$p:\Pic^0(C)\times\Div^0(C)\to \Pic^0(C)\times\Pic^0(C)$} given by the formula
\mbox{$([\Lc],E)\mapsto ([\Lc],[\OO_C(E)])$}. Then there is an isomorphism of
$k^*$-torsors $\varphi:p^*\Pc\cong \Pc'$, where
$\Pc'$ is a biextension of $(\Pic^0(C),\Div^0(C))$ by
$k^*$ given by the formula
$\Pc'_{([\Lc],E)}=(\bigotimes\limits_{x\in
C}\Lc|_x^{\otimes\ord_x(E)})\backslash\{0\}$. Thus
$\varphi$ induces a biextension structure on $p^*\Pc$;
it turns out that this structure descends to $\Pc$.
Moreover, in op.cit. it is proved that the biextension $\Pc$ is canonically isomorphic
to the Poincar\'e line bundle on $\Pic^0(X)\times\Pic^0(X)$ without the zero section.

\begin{claim}\label{claim-DeligneDC}
In the above notations, the biextensions $\Pc$ and $P_{DC}$ are canonically isomorphic.
\end{claim}

\begin{proof}
Consider a map $\pi:\Div^0(C)\times \Div^0(C)\to
\Pic^0(C)\times\Pic^0(C)$ given by the formula $(D,E)\mapsto ([\OO_C(D),\OO_C(E)])$.
There are canonical isomorphisms of $k^*$-torsors
$$
\mbox{$\pi^*\Pc\cong(\bigotimes\limits_{x\in
C}\OO_C(D)|_x^{\otimes\ord_x(E)})\backslash\{0\}\cong(\bigotimes\limits_{x\in
C}\OO_C(x)|_x^{\otimes(\ord_x(D)+\ord_x(E))})\backslash\{0\}\cong$}
$$
$$
\mbox{$\cong\langle\sum\limits_{x\in C}\ord_x(D)\OO_x,\sum\limits_{y\in
C}\ord_y(E)\OO_y\rangle.$}
$$
Besides, these isomorphisms commute with the biextension structure. It remains to note that
the rational equivalence on the first argument also commutes with the above isomorphisms
(this is a particular case of Lemma~\ref{lemma-computation}).
\end{proof}

\begin{rmk}\label{rmk-curvecase}
For a smooth projective curve, the biextension $P_{KC}$ from
Section~\ref{subsection-Kcohombiext} corresponds to a pairing
between complexes given by a certain product of Hilbert tame
symbols. Combining Proposition~\ref{prop-biext-Kcohom},
Proposition~\ref{prop-DCbiext}, and Claim~\ref{claim-DeligneDC},
we see that there is a connection between the tame symbol and the
Poincar\'e biextension given in terms of determinant of cohomology
of invertible sheaves. In \cite{Gor06} this connection is explained
explicitly in terms of the central extension of ideles on a smooth
projective curves constructed by Arbarello, de Concini, and Kac.
\end{rmk}

\subsection{Consequences on the Weil pairing}\label{subsection-weilpairing}

Let $X$ be a smooth projective variety of dimension $d$ over a field
$k$. As above, suppose that $p,q\ge 0$ are such that $p+q=d+1$.
Consider cycles $Z\in Z^p(X)'$ and $W\in Z^q(X)'$ (see
Section~\ref{subsection-acycles}) such that $lZ=\div(\{f_{\eta}\})$
and $lW=\div(\{g_{\xi}\})$ for an integer $l\in \Z$ and $K_1$-chains
$\{f_{\eta}\}\in G^{p-1}_{|W|}(X,1)$, $\{g_{\xi}\}\in
G^{q-1}_{|Z|}(X,1)$ (see Section \ref{subsection-K1}). By $[Z]$ and
$[W]$ denote the classes of the cycles $Z$ and $W$ in the groups
$CH^p(X)=H^p(X,\K_p)$ and $CH^q(X)=H^q(X,\K_q)$, respectively. As
explained in Example~\ref{examp-dg-Weilpairing}, the Massey triple
product in $K$-cohomology
$$
m_3([Z],l,[W])\in H^d(X,\K_{d+1})/([Z]\cdot H^{q-1}(X,\K_q)+
H^{p-1}(X,\K_p)\cdot[W])
$$
has a well defined push-forward $\overline{m}_3([Z],l,[W]))\in
\mu_l$ with respect to the map \mbox{$\pi_*:H^d(X,\K_{d+1})\to
k^*$}, where $\pi:X\to \Spec(k)$ is the structure morphism.

Combining Lemma~\ref{lemma-Weil-pair-commut},
Example~\ref{examp-dg-Weilpairing},
Proposition~\ref{prop-biext-Kcohom}, and
Proposition~\ref{prop-AbelJacobi}, we get the following statements.

\begin{corol}\label{corol-Weil-pairing}
\hspace{0cm}
\begin{itemize}
\item[(i)]
In the above notations, we have
$$
\overline{m}_3(\alpha,l,\beta)^{(-1)^{pq}}=\prod_{\eta}f_{\eta}(\overline{\eta}\cap
W)\cdot \prod_{\xi}g_{\xi}^{-1}(Z\cap \overline{\xi})=
\phi_l([Z],[W]),
$$
where $\phi_l$ is the Weil pairing associated with the biextension
$P_E$ from Section \ref{subsection-acycles}.
\item[(ii)]
If $k=\Cb$ and $Z\in Z^p(X)_{\hom}$, $W\in Z^q(X)_{\hom}$, then we
have
$$
\phi_l([Z],[W])=\phi^{an}_l(AJ(Z),AJ(W)),
$$
where $\phi^{an}_l$ is the Weil pairing between the
\mbox{$l$-torsion} subgroups in the dual complex tori $J^{2p-1}(X)$
and $J^{2q-1}(X)$.
\end{itemize}
\end{corol}

This is a generalization of the classical Weil's formula for
divisors on a curve. The first part of
Corollary~\ref{corol-Weil-pairing} was proved in \cite{Gor07}
without considering biextensions.

\end{document}